\documentclass[11pt]{amsart}

\usepackage{amssymb}
\usepackage{amsmath}
\usepackage{mathtools}

\usepackage{amsthm}
\usepackage{xcolor}
\usepackage{hyperref}
\usepackage[all]{xy}
\xyoption{2cell}
\usepackage{soul}

\definecolor{stef}{HTML}{0071BC}
\definecolor{stef1}{HTML}{3FBC9D}

\newcommand\cF{{\mathcal F}}
\newcommand\cG{{\mathcal G}}

\newcommand\cL{{\mathcal L}}

\newcommand\cO{{\mathcal O}}

\newcommand\cR{{\mathcal R}}

\newcommand\mP{{\mathbb P}}
\newcommand\mQ{{\mathbb Q}}
\newcommand\mR{{\mathbb R}}

\newcommand\mZ{{\mathbb Z}}

\newcommand{\cone}{\textrm{NE}}

\newcommand{\diff}{\textrm{Diff}}
\DeclareMathOperator{\vol}{\mathrm{Vol}}

\DeclareMathOperator{\sing}{\mathrm{Sing}}

\newcommand{\cal}[1]{\mathcal{#1}}

\DeclareMathOperator{\Exc}{Exc}
\DeclareMathOperator{\rk}{\mathrm rk}
\DeclareMathOperator{\supp}{Supp}

\newcommand{\rc}[2]{#1 \xymatrix{\ar@{-->}[r] & }{#2}}

\newtheorem{theorem}{Theorem}[section]
\newtheorem*{theorem*}{Theorem}
\newtheorem{lemma}[theorem]{Lemma}
\newtheorem{problem}[theorem]{Problem}

\newtheorem{proposition}[theorem]{Proposition}
\newtheorem{corollary}[theorem]{Corollary}
\newtheorem{example}[theorem]{Example}

\theoremstyle{definition}
\newtheorem{definition}[theorem]{Definition}

\theoremstyle{remark}
\newtheorem{remark}[theorem]{Remark}
\newtheorem{claim}[theorem]{Claim}
\theoremstyle{construction}
\newtheorem{construction}[theorem]{Construction}

\begin{document}
	

	\title{Explicit bounds on foliated surfaces and the Poincaré problem}
	
	\renewcommand{\theequation}{{\arabic{section}.\arabic{theorem}.\arabic{equation}}}
	
	\author{Stefania Vassiliadis}
	\address{Department of Mathematics \newline
		\indent 
		King’s College London, \newline
		\indent Strand, London WC2R 2LS, UK} 
	\email{stefania.vassiliadis@kcl.ac.uk}

	\begin{abstract}
    We give a solution to the Poincaré Problem, in the formulation of Cerveau and Lins Neto.
    We obtain a bound on the degree of general leaves of foliations of general type, which is linear in $g$. To achieve this  we study the birational geometry of foliations within the framework of the Minimal Model Program (MMP). Extending the approach of Spicer–Svaldi and Pereira-Svaldi, we study the set of pseudo-effective thresholds of adjoint foliated structures, showing that it satisfies the descending chain condition and it admits an explicit universal lower bound. These results yield effective birationality statements for adjoint divisors of the form $K_{\mathcal{F}} + \tau K_X$. 
	\end{abstract}

	\maketitle
	\section{Introduction}

We work over the field of complex numbers $\mathbb{C}$. 

Given an algebraic integrable ordinary differential equation in two variables, a central problem is to understand when it admits algebraically integrable solutions. The first attempts to answer this question date back to Poincaré \cite{poincare1891lintegration}. Indeed, he asked: given a foliation $\mathcal{F}$ on $\mathbb{P}^2$, can one bound the degree of algebraic solutions (leaves) of $\mathcal{F}$ in terms of the degree of the foliation and the genus of the solution?  
Pereira showed that the degree of a general leaf of a foliation of general type can be bounded by a function depending on the genus of the leaf and the degree of the foliation \cite{example}.

Pereira and Svaldi \cite{pereira} showed that for foliations birationally equivalent to non-isotrivial fibrations of genus $g \ge 2$, such bounds can be made explicit and depend linearly on $g$. They provide an upper bound that depends on the degree of the foliation and the genus of a general leaf, but it grows exponentially with the genus \cite[Theorem A]{pereira}.  
The crucial task in improving this bound is addressing the following problem:

\begin{problem}[{\cite[Problem 6.8]{pereira}, \cite[Problem 1]{effective}}]
Find universal bounds on $(n, m) \in \mathbb{Z}_{>0} \times \mathbb{Z}_{>0}$ ensuring the non-vanishing of
\[
h^0(X, K_{\mathcal{F}}^m \otimes K_X^n)
\]
for foliations of adjoint general type.
\end{problem}

The main contribution of our paper is to give the first explicit answer to this question.  
To tackle this problem, we study the behaviour of the set of pseudo-effective thresholds $\cR_{2, \eta,I,\varepsilon}$ (see \eqref{pseff} and \ref{set}). We show that this set satisfies the descending chain condition (DCC) (Theorem \ref{lm: DCC}) and we establish an explicit lower bound.
\begin{theorem}[=Theorem \ref{lm: DCC}]
		\label{lm: DCC intro}
		Fix two real numbers$\varepsilon,\eta>0$. Then:
		\begin{enumerate}
			\item if $I\subset [0,1]$ is a finite set, then $\cR_{2, \eta,I,\varepsilon}\cap (0,\delta)$ is a finite set for any $\delta>0$; 
			\item  if $I\subset[0,1]$ is a DCC set, then $\cR_{2, \eta,I,\varepsilon} $ satisfies the DCC.
		\end{enumerate}
		
	\end{theorem}
	
\begin{theorem}[= Theorem \ref{co: bound pseff} and Theorem \ref{lm: bound M}]
\label{thm:intro-bound}
Let $(X,\mathcal{F})$ be a foliated pair, where $X$ is a smooth surface and $\mathcal{F}$ is a canonical foliation with $K_{\mathcal{F}}$ big. Then the pseudo-effective threshold $\tau(X,\mathcal{F})$ satisfies
\[
\tau(X,\mathcal{F}) \geq \tau_0 \coloneqq \; 
        \frac{1}{3\;\Big(2\big(2\big(3(2\cdot 14^{2151296})!+1\big)\big)^{128\big(3(2\cdot 14^{2151296})!+1\big)^5}\Big)! }.
\]
Fix $0<\varepsilon<\tau_0$. Then
\[
|m(K_\cF+\varepsilon K_X)|
\]
defines a birational map for every
\[
m\geq M_0(\varepsilon)\coloneqq
\left (2+8\cdot \frac{1}{\varepsilon}\left\lfloor2\left (\frac{2}{\varepsilon}\right )^{ \frac{(2)^7}{ \varepsilon^5}  }\right\rfloor! \right)
\]
if $K_X$ is not pseudo-effective,  or
\[
m>M_0\coloneqq v\cdot 42\cdot 84^{128\cdot 42^5+168}
\]
where $v=64\cdot193^2$ if $K_X$ is pseudo-effective and $\kappa(K_X)\neq0$.

\end{theorem}
We are able to prove a version of Theorem \ref{thm:intro-bound} which allow a boundary and $\eta$-lc singularities.

While the constants above are far from optimal, they provide a universal threshold for which effective non-vanishing and birationality are guaranteed.  

As mentioned at the beginning of this paragraph, as an application, we refine a result of Pereira-Svaldi on foliations of $\mathbb{P}^2$:

\begin{theorem}[= Theorem \ref{th: bound degree}]
Let $\mathcal{F}$ be a foliation on $\mathbb{P}^2$ birationally equivalent to a non-isotrivial fibration of genus $g \geq 2$. Then the degree of a general leaf $F$ of genus $g$ satisfies
\[
  \deg F \leq 
  M_0(\tau_0) \left( \frac{1}{\tau_0} + 1 \right) (4g - 4) 
  \cdot \frac{1}{\tau_0} \deg \cF,
\]
where $M_0(\tau_0)$ and $\tau_0$ are as in Theorem \ref{thm:intro-bound}. 
\end{theorem}

We rely on the framework of adjoint foliated structures, introduced in \cite{pereira}, developed in \cite{effective} for rank-one foliations on surfaces, and in \cite{squadrone} for algebraically integrable foliations on higher-dimensional varieties.  

Adjoint foliated structures have proven useful in overcoming several issues that appear when dealing with the birational geometry of foliations. For example, effective birationality fails even for rank-one foliations on surfaces \cite{Lu21}. The general idea is that instead of studying the canonical divisor of the foliation $K_\cF$ alone, one considers divisors of the form $K_\cF+\varepsilon K_X$ for $0<\varepsilon \ll 1$. This allows one to use classical results that hold for $K_X$.

A key technical ingredient for the study of the pseudoeffective threshold is the existence of a Minimal Model Program (MMP) for adjoint foliated triples on surfaces:

\begin{theorem}[= Theorem \ref{th: adjoint doppia foliazione}]
Fix $\varepsilon>0$.
Let $(X, \cF,\Delta)$ be a foliated triple, where $X$ is a normal projective $\mathbb{Q}$-factorial surface, and $\Delta\geq 0$ is a boundary. Suppose moreover that $(\cF,\Delta^{n-inv})$ is log canonical or $(\cF,\Delta^{n-inv})$ is a foliated pair where $\cF$ is non-dicritical. Then there exists a
$K_{(X,\cF,\Delta)_\varepsilon}$-MMP $\rho:X\to Y$.

Let $\cF'$ be the induced foliation on $Y$ and $\Delta'\coloneqq \rho_* \Delta$. Then the following properties hold:
\begin{enumerate}
\item if $K_{(X,\cF,\Delta)_\varepsilon}$ is pseudo-effective, then $K_{(Y,\cF',\Delta')_\varepsilon}$ is nef;
\item if $K_{(X,\cF,\Delta)_\varepsilon}$ is not pseudo-effective, then $Y$ has a contraction of fibre type $\pi:Y\to Y'$ with $\rho(Y/Y')=1$ such that $-K_{(Y,\cF',\Delta')_\varepsilon}$ is $\pi$-ample.
\end{enumerate}

Moreover, such an MMP factors as
\[\xymatrix{
X\ar[r]^{f_0}&X_0\ar[r]^{f_1}& X_1\coloneqq Y}
\]
where:
\begin{enumerate}
\item[(A)] $f_0$ is a partial $K_{(X,\cF,\Delta)_\varepsilon}$-MMP that contracts only $K_{\cF}+\Delta^{n-inv}$-non-positive curves. Denote $\cF_0\coloneqq f_{0*}\cF$ and $\Delta_0\coloneqq f_{0*}\Delta$;
\item[(B)] $f_1$ is a partial $K_{X_0}+\Delta_0$-MMP which only contracts curves that are $K_{X_0}+\Delta_0$ and $K_{(X_0,\cF_0,\Delta_0)_\varepsilon}$-negative and $K_{\cF_0}+\Delta_0^{n-inv}$-positive, where $\cF_1 \coloneqq f_{1*} \cF_0$ and $\Delta_1 \coloneqq f_{1*} \Delta_0$.
\end{enumerate}

Moreover, the following hold:
\begin{enumerate}
\item[(3)] if $(X,\Delta)$ is log canonical and $(\cF,\Delta^{n-inv})$ is log canonical, then $(\cF_0,\Delta_0)$ is log canonical and $(X_i,\Delta_i)$ is log canonical for $i=0,1$;
\item[(4)] for $i=0,1$, if $(X,\Delta)$ is $\eta$-lc for some $\eta>0$ and $(\cF,\Delta^{n-inv})$ is log canonical, then $(X_i,\Delta_i)$ is $\eta'$-lc for $\eta'\coloneqq\frac{\varepsilon\eta}{\varepsilon+1}$;
\item[(5)] if the foliated triple $(X,\cF, \Delta)$ has $\varepsilon$-adjoint log canonical singularities, then at each step of the $K_{(X,\cF,\Delta)_\varepsilon}$-MMP, $(X_i,\cF_i,\Delta_i)$ has $\varepsilon$-adjoint log canonical singularities.
\end{enumerate}

\end{theorem}

 \medskip
{\bf Acknowledgments.}  
I would like to thank my PhD advisor Calum Spicer for suggesting this project, for constant guidance and for many fruitful discussions. I would also like to thank Marta Benozzo, Federico Bongiorno, Riccardo Carini, Samuele Ciprietti, Roktim Mascharak,  Carla Novelli, Sara Veneziale, and Pascale Voegtli
for helpful advice and discussions. This work was supported by the Engineering and Physical Sciences Research Council [EP/S021590/1] and the
EPSRC Centre for Doctoral Training in Geometry and Number Theory (The London School of Geometry and Number Theory), University College London. This research was conducted at King’s College
London and Imperial College London.

\section{Preliminaries}

\subsection{Foliations}
\begin{definition}
Let $X$ be a normal variety. A \textit{foliation} $\cF \subset TX$ is a saturated subsheaf of the tangent sheaf of $X$ which is closed under the Lie bracket.
The \textit{rank} of the foliation, denoted by $\rk \cF$, is the rank of $\cF$ as a sheaf, and its \textit{co-rank} is $\dim X - \rk \cF.$
\end{definition}

\begin{definition}[\cite{brunella2015birational}]
Let $\cF$ be a foliation on $\mP^2$. We define the \textit{degree of $\cF$} as the number of tangencies of the foliation with a generic line. Moreover, the following holds:
\[
K_\cF = \cO_{\mP^2}(d-1).
\]
\end{definition}

\begin{definition}[\cite{druel}, Section 3.2; \cite{corank}, Section 2.3]
Let $X$ be a normal variety and let $\cF$ be a foliation on $X$. Let $\varphi: Y \dashrightarrow X$ be a dominant map and assume that there exist smooth open subsets $U \subset X$
and $V \subset Y$ such that the restriction $\varphi_{|_V} : V \to U$ is a morphism.
Denote by $\cF_U$ the restriction of $\cF$ to $U$.
Then the morphism $N^*_{\cF_U}\to \Omega^1_U$
induces a morphism $(\varphi_{|_V})^*N^*_{\cF_U} \to \Omega^1_V$, and hence a foliation $\cG_V$ on $V$.
We then extend $\cG_V$ to a foliation on $Y$, which we refer to as the
\textit{induced foliation} on $Y$ by $\varphi$. If $\varphi$ is a morphism, we call the induced foliation the \textit{pull-back foliation}
and denote it by $\varphi^{-1}_*\cF$.
\end{definition}

In this note, we work with foliations on normal projective surfaces, and all our foliations will have rank one (equivalently, co-rank one). Therefore, we can apply all the results of \cite{corank,rank-one}.

Given a foliation $\cF$ on $X$, we obtain a subsheaf of the cotangent sheaf by considering the kernel of 
\[
\Omega_X \to \cF^*.
\]
Let $\cL$ be this kernel; we then define the conormal sheaf of the foliation as $\cL$ in $\Omega_X$. We denote the conormal sheaf by $N^*_\cF$.
Hence, on the smooth locus of $X$, a foliation can be described by $1$-forms whose zero loci have codimension at least $2$.

Consequently, we obtain the following equality of Weil divisors:
\[
K_X = K_\cF - N_\cF.
\]

\subsection{Pairs and Triples}

Let $X$ be a normal variety and let $\cF$ be a foliation on $X$. Let $D$ be an $\mR$-divisor; we define the \emph{$\cF$-invariant part of $D$}, denoted by $D^{inv}_\cF$, as the part of $D$ with $\cF$-invariant support, and $D^{n-inv}_\cF$ as the part of $D$ with no $\cF$-invariant components. Note that $D$ uniquely decomposes as $D = D^{inv}_\cF + D^{n-inv}_\cF$. If there is no risk of confusion, we will drop the dependence on $\cF$ in the notation.

We say that $(X,\Delta)$ is a \emph{pair} if $\Delta$ is an effective $\mR$-divisor and $K_X + \Delta$ is $\mR$-Cartier.
A \emph{foliated pair} $(\cF,\Delta)$ on $X$ is the data of a foliation $\cF$ on $X$ and an effective $\mR$-divisor $\Delta$ such that $K_{\cF}+\Delta$ is $\mR$-Cartier.

\begin{definition}[cf.~\cite{effective}, Section 2.2]
A \emph{foliated triple} $(X,\cF,\Delta)$ consists of a foliation $\cF$ on a normal variety $X$, and an effective $\mR$-divisor $\Delta \geq 0$ such that both $K_\cF + \Delta^{n-inv}_\cF$ and $K_X + \Delta$ are $\mR$-Cartier.
\end{definition}

\begin{remark}
All our results involving foliated triples could equivalently be stated using the notation of adjoint foliated structures \cite[Definition 1.2]{squadrone}, in the case where $M = 0$. 
\end{remark}

\begin{definition}
Let $D$ be a divisor on a surface $X$. We say that a contraction $\varphi: X \to Y$ is \emph{$D$-negative} if it is a contraction such that $D \cdot C < 0$ for every curve $C$ contracted by $\varphi$. We say that $\varphi$ is \emph{elementary} if $\rho(X/Y) = 1$.
\end{definition}
\begin{remark}
    By \cite{mum} the intersection for Weil divisor in the above definition is well defined. 
\end{remark}

\begin{definition}[cf.~\cite{casagrande2013birational}, Section 2.10]
Let $D$ be a divisor on a normal $\mQ$-factorial surface $X$.
A \emph{Mori Program for $D$ (or $D$-MMP)} is a sequence
\[
\xymatrix{
X \coloneqq X_0 \ar[r]^{\rho_0} & X_1 \ar[r]^{\rho_1} & X_2 \ar[r] & \cdots \ar[r] & X_{t-1} \ar[r]^{\rho_{t-1}} & X_t
}
\]
such that:
\begin{enumerate}
\item each $X_i$ is a normal $\mQ$-factorial surface, and $D_{i+1} \coloneqq \rho_i^{-1} D_i$;
\item each $\rho_i$ is a birational, $D_i$-negative elementary contraction of divisorial type for $i = 0, \ldots, t-1$;
\item $D_t$ is nef, or $X_t$ admits a $D_t$-negative elementary contraction of fibre type.
\end{enumerate}
\end{definition}

\subsection{Singularities}

For the classical notions of singularities of pairs, we refer to \cite{kollár_mori_1998}. We denote by $X^{reg}$ the smooth locus of $X$.  

Let $(\cF,\Delta)$ be a foliated pair on $X$ and let $D$ be a divisor over $X$.
We define the number
\[
i_\cF(D) =
\begin{cases}
0 & \text{if } D \text{ is } \cF\text{-invariant},\\
1 & \text{if } D \text{ is } \cF\text{-non-invariant.}
\end{cases}
\]
In the literature, $i_\cF(D)$ is often denoted by $\varepsilon_{\cF}(D)$. When there is no risk of confusion, we omit the dependence on $\cF$ in the notation.

Consider a birational morphism $\pi: Y \to X$, and let $\cF_Y$ be the pulled-back foliation on $Y$.
We write
\[
K_{\cF_Y} = \pi^*(K_\cF + \Delta) + \sum a(\cF,\Delta,E)E,
\]
where the sum runs over all prime divisors on $Y$ and
\[
\pi_*\!\left(\sum a(\cF,\Delta,E)E\right) = -\Delta.
\]
We say that $(\cF,\Delta)$ is \textit{terminal} (resp.~\textit{canonical}, \textit{log terminal}, \textit{log canonical})
if $a(\cF,\Delta,E) > 0$ (resp.~$\ge 0$, $> -\iota(E)$, $\ge -\iota(E)$) for every birational morphism $\pi: Y \to X$ and every prime divisor $E \subseteq Y$, where
$\iota(E) = 0$ if $E$ is $\cF$-invariant and $\iota(E) = 1$ otherwise.

\begin{definition}[cf.~\cite{HD}]
Given a foliation $\cF$ on a normal surface $X$, we say that $\cF$ has \textit{non-dicritical} singularities if for any sequence of blow-ups
\[
\pi: (X',\cF') \to (X,\cF)
\]
such that $X'$ is smooth and for any $Q \in X$, the fibre $\pi^{-1}(Q)$ is invariant under $\cF'$. The foliation is \textit{dicritical} if it is not non-dicritical.
\end{definition}

\begin{remark}
Let $P \in X$ be a dicritical singularity. An F-dlt modification, cf.~\cite[Theorem 1.4]{corank}, will always extract at least one exceptional divisor transverse to the foliation, since an F-dlt foliation is non-dicritical, cf.~\cite[Theorem 8.1]{corank}.
\end{remark}

\subsection{$\varepsilon$-adjoint log canonical divisors and $\varepsilon$-adjoint log canonical foliated singularities}

In \cite{effective}, the authors introduced and studied the notion of the $\varepsilon$-adjoint log canonical divisor, which encodes the properties and behaviour of a foliated triple $(X,\cF,\Delta)$ under birational transformations.

\begin{definition}
\label{def: adjoint}
Let $(X,\cF,\Delta)$ be a foliated triple on a normal projective variety $X$. Fix $\varepsilon > 0$. The \textit{$\varepsilon$-adjoint divisor} of $(X,\cF,\Delta)$ is defined as
\[
K_{(X,\cF,\Delta)_\varepsilon} \coloneqq K_\cF + \Delta^{n-inv} + \varepsilon (K_X + \Delta^{n-inv}).
\]
\end{definition}

We can measure the singularities of the triple $(X,\cF,\Delta)$ in terms of the behaviour of $K_{(X,\cF,\Delta)_\varepsilon}$ under birational maps.

\begin{definition}
Let $(X,\cF,\Delta)$ be a foliated triple on a normal projective variety $X$, and fix $\varepsilon > 0$.  
We say that $(X,\cF,\Delta)$ is \textit{$\varepsilon$-adjoint log canonical} (resp.~\textit{$\varepsilon$-adjoint klt}) if for every birational morphism $\pi: X' \to X$,
\[
K_{(X',\cF',\Delta')_\varepsilon} = \pi^*(K_{(X,\cF,\Delta)_\varepsilon}) + E,
\]
where $E = \sum a_i E_i$ is $\pi$-exceptional, $\Delta' \coloneqq \pi^{-1}_*\Delta$, and
\[
a_i \ge -\big(i_\cF(E_i) + \varepsilon\big) \quad (\text{resp. } \lfloor \Delta \rfloor = 0 \text{ and } a_i > -\big(i_\cF(E_i) + \varepsilon\big)).
\]
\end{definition}

\begin{lemma}
\label{lm: lc}
Let $0<\varepsilon< \frac{1}{5}$. Let $(X,\cF,\Delta)$ be a foliated triple on a klt projective surface with $\Delta^{n-inv} = 0$, which is $\varepsilon$-adjoint log canonical. Then $\cF$ is log canonical.
\end{lemma}

\begin{proof}
Assume for contradiction that $\cF$ is not log canonical at a point $P$. If $X$ is smooth at $P$, we conclude by \cite[Lemma 2.19]{effective}. Suppose instead that $X$ is singular at $P$, then $P$ is a quotient singularity of $X$.
Consider the index-one cover
\[
\mu: \overline{X} \to X,
\]
and denote by $\overline{\cF}$ the pull-back foliation on $\overline{X}$, and set $\overline{\Delta} \coloneqq \pi_*^{-1}\Delta$. By \cite[Lemma 2.20.(1)]{effective}, $(\overline{X},\overline{\cF},\overline{\Delta})$ is $\varepsilon$-adjoint log canonical. By \cite[Lemma 2.19]{effective}, $\overline{\cF}$ is log canonical. We then conclude by \cite[Lemma 2.20.(2)]{effective}.
\end{proof}

\begin{proposition}
\label{pr; rimane lc}
Fix $0<\varepsilon< \frac{1}{5}$.
Let $\cF$ be a log canonical foliation on a klt surface $X$, and let $\Delta \ge 0$ be a boundary with $\Delta^{n-inv} = 0$. Suppose that $(X,\cF,\Delta)$ is $\varepsilon$-adjoint lc.
Let $\varphi: X \to X'$ be a partial $K_{(X,\cF,\Delta)_\varepsilon}$-MMP, and let $\cF'$ be the induced foliation on $X'$.
Then, $\cF'$ is log canonical.
\end{proposition}

\begin{proof}
By  \cite[Lemma 3.39]{kollár_mori_1998}, $(X',\cF')$ is $\varepsilon$-adjoint log canonical. We conclude using Lemma~\ref{lm: lc}.
\end{proof}

We will use the following result on partial resolutions of strictly log canonical centres.

\begin{lemma}[\cite{effective}, Lemmata 2.10 and 2.11]
Let $X$ be a normal $\mQ$-factorial klt surface, and let $\cF$ be a foliation on $X$. Let $P \in X$ be a strictly log canonical centre of $\cF$.
Then, there exists a resolution of $X$ that extracts only divisors with discrepancy $\le -\iota(E)$ and for which the induced foliation is canonical.
\end{lemma}
\section{Some preliminary results}
	\subsection{Adjunction Theory and applications}
    
	The foliated adjunction formula relates the  divisors $(K_\cF+\varepsilon(D)D)|_D$ and $K_{\cF_D}$ where $\varepsilon(D)$ is 0 if $D$ is $\cF$-invariant and 1 otherwise.
	In general, these two divisors are not linearly equivalent. As in the classical case, we need to add a correction term $\Theta'\geq 0$ called \textit{foliated different}. 
	
    We have the following foliated analogue of the classical adjunction (see \cite{HD} or \cite{adjunction}).
	Let $X$ be a normal surface, let $\mathcal F$ be a foliation on $X$, let $D \subset X$ be an irreducible divisor and let $n: D^n\to D$ be the normalization of $D$. Then on $D^n$ there is a canonically defined restricted foliation, denoted $\mathcal F_{D_n}$, and we have an equality 
	\[ n^*(K_\cF+\varepsilon(D)D)=K_{\cF_{D^n}}+\Theta'\]
	where $\Theta' \ge 0$.

    When $D$ is $\mathcal F$ invariant then $\mathcal F_{D^n}$ is the foliation with one leaf, i.e., $T_{\mathcal F_{D^n}} = T_{D^n}$, and
    in the case that $D$ is not $\mathcal F$-invariant then $\mathcal F_{D^n}$ is the foliation by points, i.e., $T_{\mathcal F_{D^n}} = 0$.

	\begin{lemma}\cite[Lemma 8.9]{HD}
		\label{lm: disc foliaz uguale inv}
		Let $(X,\cF,\Delta)$ be a foliated triple where $X$ is a surface with klt singularities and $(\cF,\Delta)$ is a foliated pair where $\cF$ is a rank one foliation and $(\cF,\Delta)$ has canonical singularities and terminal singularities at every point of $\sing(X)$. Let $C$ be an $\cF$-invariant curve, and let $n:C^n\to C$ be the normalization.
		
		By adjunction we may write \[n^*(K_X+C)=K_{C^n}+\Theta\] and  \[n^*K_\cF=K_{C^n}+\Theta'.\] Then $\Theta' -\Theta\geq 0,$ with equality along those centres contained in $\sing X$.
          
		Moreover,  $$\diff(\cF,\Delta)\geq \diff(X,\Delta+C)$$
		where $\diff(\cF,\Delta)=\Theta'+n^*(\Delta)$.
	\end{lemma}
	
	\begin{remark}
		\label{lm adj invariant}
		Let $(X,\cF,\Delta)$ be a foliated triple where $X$ is a surface with klt singularities and $(\cF,\Delta)$ is canonical and has terminal singularities at every point of $\sing(X)$.
		Let $C$ be an $\cF$-invariant curve such that $C\not\subset \supp(\Delta)$ and let $n:C^n\to C$ be the normalization. Then by Lemma \ref{lm: disc foliaz uguale inv} the following holds
		\[n^*(K_\cF) =K_{C^n}+\sum_{i=0}^{t}a_iP_i+\sum_{j=0}^{s}\frac{n_j-1}{n_j}Q_j\]
		where $a_j$ are positive natural number for every $j$, $n_j\in \mZ$, $n_j\geq 2$ $Q_j\in X$ are inverse image under $n$ of the singular points of $X$ on $C$ and $P_i$ are the inverse image under $n$ of singular points of $\cF$ in $C\cap (X \setminus \sing(X))$. 
		Hence we obtain the following equality:
		\[\begin{split}
			n^*(K_\cF +\Delta+\varepsilon (K_X+\Delta))= &K_{C^n} +\sum_{i=0}^{t}a_iP_i+\sum_{j=0}^{s}\frac{n_j-1}{n_j}Q_j+n^*(\Delta)+\\
			& +\varepsilon \left (K_{C^n}-{C^2}+\sum_{j=0}^{s}\frac{n_j-1}{n_j}Q_j+n^*(\Delta)\right )~.
		\end{split}\]

	\end{remark}

	The next result is proved for rank one foliations on surfaces in \cite[Proof of lemma 2.7]{effective}; using \cite{adjunction} we obtain the same results for rank one foliations on varieties of any dimension.
	\begin{corollary}
		\label{lm: un punto singolare}
		Let $X$ be a normal klt, $\mQ$-factorial projective variety, $\cF$ a rank one foliation on $X$ and let $C$ be an $\cF$-invariant curve such that:
		\begin{enumerate}
			\item $K_\cF\cdot C<0$;
			\item $C$ contains a point $P$ where $\cF$ is worse than canonical.
		\end{enumerate}
		Then $\cF$ is terminal at every point $Q\in C \setminus \{P\}$.
        \end{corollary}

	\begin{lemma}
		\label{lm: C movable}
		Let $X$ be a normal $\mQ$-factorial projective surface, let $\cF$ be a rank one foliation on $X$ and let $C$ be an $\cF$-invariant curve such that $K_\cF\cdot C<0$.
		Moreover, suppose that the following hold: let 
		$\varphi:Y\to X $ be an F-dlt modification, cf. \cite[Theorem 1.4]{corank}, write $K_\cG+\sum_jb_jE_j=\varphi^*K_\cF$ where $E_j$ are the $\phi$-exceptional divisors, and let $\widetilde{C}$ be the strict transform of $C$ in $Y$. Suppose that there exists $j$ such that $E_j\cap \widetilde{C}\neq \emptyset$ and $b_j\geq 1.$
		Then $C$ is movable.
	\end{lemma}
	This proof follows the ideas of \cite[Proof of Lemma 2.7]{effective}.
	\begin{proof}
		By Corollary \ref{lm: un punto singolare}, we can replace $X$ by a smaller open neighbourhood of $C$, where $\cF$ is worse than canonical only at some point  $P \in C$ and it is terminal at every point of of $C \setminus P$. Let $\varphi:Y\to X$, be the F-dlt modification as in the statement. Let $\cG$ be the induced foliation on $Y $  and let $\widetilde C$ be the strict transform of $C$. Note that $\widetilde C$ is $\cG$-invariant and $K_\cG\cdot \widetilde C <0$. 
		We claim that the foliation $\cG$ is terminal at every point of $\widetilde C$. 
		 Note that $\cG$ is non-dicritical and $(\cG, \sum \varepsilon(E_i)E_i)$ is log canonical. Let $F$ be any exceptional divisor centred over a point of $E\cap \widetilde C$, then we have that $a(F,\cG,\sum \varepsilon(E_i)E_i)\geq \varepsilon(F)=0$. 
		Since $\cG$ has F-dlt singularities, then by \cite[Lemma 3.8]{corank} there are two possibilities:
		\begin{enumerate}
			\item $\cG$ is terminal at $E\cap \widetilde{C}$; or
			\item $E\cap \widetilde{C}$ is a smooth point for $Y$.
			
		\end{enumerate}
		Suppose that we are in the first case, then  $\cG$ is terminal on $\widetilde C.$ By \cite[Proposition 3.3]{rank-one}, we see that $\widetilde C$  moves, hence $C$ moves.
		In the second case, since $P$ is a smooth point for $Y$, then $\cG$ is Gorenstein. 
		By adjunction
		\[K_\cF|_C= K_C+\diff(\cF,0).\] 
		Note that $E\cap \widetilde{C}$ is a smooth point for $Y$ but it is a canonical singularity for $\cG$. Hence, $E\cap \widetilde{C}$ appears in $\diff(\cF,0)$ with coefficient $a>0$.  Since $Y$ is smooth, then  $a\geq 1$, which means that $K_\cG\cdot \widetilde{C}\geq -1$.
		Since $P$ is a log canonical singularity or worse, then 
		\[K_\cG+\sum b_jE_j=\varphi^*K_\cF.\]
		Recall we have assumed, $E_j\cap \widetilde{C}\neq \emptyset$ and $b_j\geq 1.$ Since $K_\cG\cdot\widetilde{C}<1$, then by intersecting with $\widetilde{C}$ we obtain 
		$(K_\cG+\sum b_jE_j)\cdot \widetilde{C}\geq 0$,
		a contradiction.
	\end{proof}

	\begin{lemma}
		\label{lm: invariant}
		Let $(X,\Delta)$ be a pair where $X$ is a normal projective $\mQ$-factorial surface and $\Delta\geq 0$ is a boundary,  $(\cF,\Delta=\Delta^{n-inv})$ a foliated pair where $\cF$ is a rank one foliation. Let $C$ be a curve such that:
		\begin{enumerate}
			\item $(K_\cF+\Delta)\cdot C<0$;
			\item $C^2<0$;
            \item $C$ is $\cF$-invariant.
		\end{enumerate}
        Let $\varphi:Y\to X$ be an F-dlt modification of $(\cF,\Delta)$. Suppose that $a(\cF,E)\leq a((X,C), E)$  for every $\varphi$-exceptional divisor over $C$ such that  $\iota(E)=1$.
        Then $(K_X+\Delta+C)\cdot C<0$.
			\end{lemma}
            \begin{remark} 
                With the notation of Lemma \ref{lm: invariant},
                if $(\cF,\Delta^{n-inv})$ is log canonical, then $a(\cF,E)\leq a((X,C), E)$  for every $\varphi$-exceptional divisor over $C$ such that  $\iota(E)=1$.
            \end{remark}
		\begin{proof}
            If $(K_\cF,\Delta)$ is log canonical in a neighbourhood of $C$, then we can conclude by \cite[Lemma 2.7]{effective}. So, suppose that $(\cF,\Delta)$ is not log canonical in a neighbourhood of $C$.
            Then there exists a point $P\in C$ where $(\cF,\Delta)$ is worse than log canonical. Since $(K_\cF+\Delta)\cdot C<0,$ then $P$ is the unique point on $C$ where $(\cF,\Delta)$ is not log canonical. Moreover $\cF$ is terminal on $C\setminus P.$
            	
			Let $\varphi:X'\to X$ be a F-dlt resolution,  then we can write $$K_\cG+\widetilde{\Delta}+\sum_jb^0_jE^0_j+\sum_jb^1_jE^1_j=\varphi^*(K_\cF+\Delta),$$
			$$K_{X'}+\widetilde{\Delta}+\widetilde{C}+\sum_ja^0_jE^0_j+\sum_ja^1_jE^1_j=\varphi^*(K_X+\Delta+C)$$
			 where $\mathcal G = \phi^{-1}\mathcal F$, $\widetilde{C}$ is the strict transform of $C$, $\widetilde{\Delta}$ is the strict transform of $\Delta$, $E^0_j$ are invariant $\varphi$-exceptional divisors, and $E_j^1$ are non-invariant $\varphi$-exceptional divisors.

			 Then, $$K_\cG+\tilde\Delta+\sum_jE^1_j=\varphi^*(K_\cF+\Delta)+\sum_j(1-b^1_j)E^1_j+\sum_j(-b^0_j)E^0_j,$$
			 $$K_{X'}+\widetilde{\Delta}+\widetilde{C}+\sum_jE^0_j+\sum_jE^1_j=\varphi^*(K_X+\Delta+C)+\sum_j(1-a^1_j)E_j+\sum_j(1-a^0_j)E^0_j.$$

			 Since $\varphi$ is an F-dlt modification, then $b_j^1\geq 1$ for every $j$.
             Therefore by Lemma \ref{lm: C movable},  $\widetilde{C}\cap E^1_j=\emptyset$ for every $j$. 
            Since $a(\cF,E)\leq a((X,C), E)$  for every $\varphi$-exceptional divisors with $\iota(E)=1$.
            Hence $(K_\cG+\tilde\Delta-(K_X'+\widetilde{\Delta}+\widetilde{C}+\sum_jE^0_j))\cdot E_j^1\geq 0.$
             By \cite[Lemma 8.9]{HD} conclude that $K_\cG+\tilde\Delta-(K_X'+\widetilde{\Delta}+\widetilde{C}+\sum_jE^0_j)$ is $\varphi$-nef. Therefore by the \textit{Negativity Lemma} $-b_j^0-1+ a^0_j<0$ for every $j$. Hence $-1+a^0_j<b^0_j$. Therefore
              
               \begin{equation}
			 	\begin{split}
			 		0>&(K_\cG+\sum b_j^1E_j^1+\sum b^0_jE_j^0)\cdot \widetilde{C}\\
                    \geq&(K_X'+\widetilde{\Delta}+\widetilde{C}+\sum b_j^1E_j^1+\sum_jE^0_j)\cdot \widetilde{C}+\sum b^0_jE_j^0\cdot\widetilde{C}\\
			 		\geq&(K_X'+\widetilde{\Delta}+\widetilde{C}+\sum b_j^1E_j^1+\sum_jE^0_j)\cdot \widetilde{C}+(-1+a_j^0)E_j^0\cdot\widetilde{C}\\
			 		\geq&(K_X+\Delta+C)\cdot C.
			 	\end{split}
			 \end{equation}
			 
             \end{proof}
             We will later use the following result on the existence of partial resolutions along log canonical centres.
               \begin{lemma}\cite[Lemma 4.4]{effective}
\label{lem_resolution_logsmooth_near_lc}
Let
$h \colon (\cal Z, \cal L, \Xi) \to T$
be a bounded family of 2-dimensional projective foliated triples $(Z_t, \cal L_t, \Xi_t)$.
Assume that for all 
$t \in T$, 
$\Xi_{t, n-inv} = \Xi_t$ and
$(\cal L_t, \Xi_t)$ is log canonical.

Passing to a stratification of 
$T$ 
into locally closed sets, and a finite covering of the irreducible components of the stratification, there exists a bounded family of 2-dimensional projective foliated triples 
$j \colon (\cal Y, \cal G, \Gamma) \to T$
and a birational morphism over $T$,
$g \colon \cal Y \to \cal Z$,
such that for all $t \in T$,
\begin{enumerate}
\item 
$\cal G_t \coloneqq g^{-1}\cal L_t$ 
and 
$\Gamma_t \coloneqq g_\ast^{-1}\Xi_t$;

\item 
$(Y_t, \Gamma_t)$ 
is log smooth in a neighborhood of 
$g_t^{-1}(P)$ 
where $P$ is a strictly log canonical point of 
$(\cal L_t, \Xi_t)$;

\item 
$g_t$ 
only extracts divisors of discrepancy (resp. foliation discrepancy) 
$\leq 0$ 
(resp. 
$=-\iota(E)$);

\item 
any foliated log resolution 
$\tau_t \colon \overline{Z}_t \rightarrow Z_t$ 
of 
$(Z_t, \cal L_t, \Xi_t)$
factors as 
\begin{align*}
\xymatrix{
\overline{Z}_t \ar[r] \ar@/_/[rr]_{\tau_t}& Y_t \ar[r]^{g_t} & Z_t.
}
\end{align*}
\end{enumerate}

\end{lemma}

        \subsection{Singularities of the underlying variety during an adjoint MMP}
             In \cite[Corollary 3.3]{effective} it is shown that the singularities that arise while running a $K_{(X,\cF,\Delta)_\varepsilon}$-MMP fixing $0<\varepsilon\ll1$ starting form a log smooth pair $(X,\Delta)$ are mild.
                In the next section we show that the same holds if $(X,\Delta)$ is $\eta$-lc for some $\eta>0$ and $0<\varepsilon\leq1$ (cf. Proposition \ref{cor: sing}).

             \begin{lemma}
\label{lm: zariski}
    Let $D_1$ and $D_2$ be two divisors on $X$ a surface and suppose that we can run a $D_1+D_2$-MMP and a $D_2$-MMP without restrictions on the singularities of $X$, denoted $\varphi\colon X \to Y$. Then we can factor this MMP as:
 \[\xymatrix{
		X\ar[r]^{f_0}&X_0\ar[r]^{f_1}& X_1\coloneqq Y}
	\]
    where
    
    \begin{enumerate}
        \item $f_0$ is a $D_1$-negative contraction; and 
        \item $f_1$ contracts $D_1$-positive and $D_2$-negative curves.
    \end{enumerate}
\end{lemma}
\begin{proof}

Up to perturbing with $D_1$ or $D_2$ with a sufficiently high multiple of the pull back of an ample divisor on $Y$, we can always assume $D_1$ and $D_2$ are pseudo-effective.
    Let $D_1=N_1+P_1$, $D_2=N_2+P_2$, and $D_1+D_2=P+N$  be the Zariski decompositions of $D_1$, $D_2$, and $D_1+D_2$ respectively. Note that $N\leq N_1+N_2.$ First, we can contract all the curves that are in $N_1$ and in $N$, denote such contraction with $f_0:X\to X_0$. By construction, $f_0$ contracts $D_1$-negative curves. If $f_{0*}N=\emptyset$, we can conclude. Suppose from now on that $f_{0*}N\neq\emptyset$. 
    Let $f_1:X_0\to X_1$ be the contraction on $f_{0*}N$. The only thing left to prove is that $f_1$ is $D_2$-negative.
    Suppose by contradiction that $f_1$ is not $D_2$-negative, then there exists a curve $C$, $D_1$-negative. Since $f_{0*}P$ and $f_{0*}P_1$ are nef, then $C$ is in $f_{0*}N\cap f_{0*}N_1,$ a contradiction. 
 \end{proof}

              \begin{lemma}
        \label{lm: sing}
            Fix positive real numbers $\varepsilon, \eta \in \mathbb R_{>0}$ and a DCC set $I$.
            Let $(X,\cF,\Delta)$ be foliated triple where $X$ is a projective $\mQ$-factorial surface and $\cF$ is a rank one foliation, $\Delta\in I$. Suppose moreover that 
            \begin{enumerate}
                \item $(X,\Delta)$ is $\eta$-lc, 
                \item $(\cF,\Delta^{n-inv})$ is log canonical, and 
                \item $(X,\cF,\Delta)$ has $\varepsilon$-adjoint lc singularities.
                \end{enumerate}
                Let $\varphi:(X,\cF,\Delta)\to (Y,\cG=\varphi_*\cF,\Delta'=\varphi_*\Delta)$ be a partial $K_{(X,\cF,\Delta)_\varepsilon}$-MMP that is $K_{\cF}+\Delta^{n-inv}$-negative.
            Then $(Y,\Delta')$ is $\eta'$-lc where $\eta'=\frac{\varepsilon\eta}{\varepsilon+1}$.
        \end{lemma}
        \begin{proof}
            The proof follows \cite[Corollary 3.3]{effective}. For convenience we give a sketch of a proof.
            We proceed by induction the number of elementary contractions of $\varphi$.
            	Write $\varphi$ as 
		
		\[
		\xymatrix{X\coloneqq X_0\ar[r]^{\varphi_1}& X_1\ar[r]^{\varphi_2}&	X_2\ar[r]&	\cdots\ar[r]&X_{n-1}\ar[r]^{\varphi_{n}}&	X_n
		}	
		\] 
        where each $\varphi_i$ is an elementary contraction.
		We denote $X_0\coloneqq X$, $\cF_0\coloneqq \cF$, $\Delta_0\coloneqq \Delta$, $\Delta_i\coloneqq \varphi_{i*}\Delta_{i-1}$  $\cF_i\coloneqq \varphi_{i*}\cF_{i-1}$, $X_n\coloneqq Y$, and $\cF_n\coloneqq\cG$. 
		By Theorem \ref{th: adjoint doppia foliazione}, $(X_i,\cF_i,\Delta_i)$ is $\varepsilon$-adjoint log canonical for every $i$. 
        For $n=0, $ the statement holds. Indeed by hypothesis $\Delta\in[0,1-\eta]$ and $\eta>\frac{\epsilon \eta}{1+\varepsilon}$, hence $(X,\Delta)$ is $\frac{\epsilon \eta}{1+\varepsilon}$-log canonical. 
		
		We can assume that $n>0$. Suppose that $(X_i,\Delta_i)$ is $\frac{\eta\varepsilon}{1+\varepsilon}$-log canonical for $i=0,\cdots,n-1.$ 
        We need to check that $(X_n,\Delta_n)$ is $\frac{\eta\varepsilon}{1+\varepsilon}$-log canonical at every point $P\in X_n$ where $\varphi$ is not an isomorphism.
        Since $\varphi$ is $(K_{\cF}+\Delta^{n-inv})$-negative, then $(\cF_i,\Delta_i^{n-inv})$ is log canonical for every $i.$ Therefore we can conclude the proof as in \cite[Corollary 3.3]{effective}.
        \end{proof}
        \begin{proposition}
         \label{cor: sing}
            Fix positive real numbers $\varepsilon, \eta \in \mathbb R_{>0}$ and a DCC set $I$.
            Let $(X,\cF,\Delta)$ be foliated triple where $X$ is a projective $\mQ$-factorial surface and $\cF$ is a rank one foliation, $\Delta\in I$. Suppose moreover that 
            \begin{enumerate}
                \item $(X,\Delta)$ is $\eta$-lc, 
                \item $(\cF,\Delta^{n-inv})$ is log canonical, and 
                \item $(X,\cF,\Delta)$ has $\varepsilon$-adjoint lc singularities.
                \end{enumerate}
                Let $\varphi:(X,\cF,\Delta)\to (Y,\cG=\varphi_*\cF,\Delta'=\varphi_*\Delta)$ be a partial $K_{(X,\cF,\Delta)_\varepsilon}$-MMP.
            Then $(Y,\Delta')$ is $\eta'$-lc where $\eta'=\frac{\varepsilon\eta}{\varepsilon+1}$.
        \end{proposition}
	\begin{proof}
	    The proof follows from Lemmata \ref{lm: sing} and \ref{lm: zariski}.
	\end{proof}

	\section{Adjoint MMP}
	The aim of this section is to prove the existence of the $K_{(X, \cF, \Delta)_\varepsilon}$-MMP (cf. Theorem \ref{th: adjoint doppia foliazione}) without restricting $\varepsilon$. This adjoint MMP was introduced in \cite{effective} where they prove the existence of the $K_{(X,\cF,\Delta)_\varepsilon}$-MMP when $\varepsilon>0$ is sufficiently small. 
	Here we are able to prove the existence of an $K_{(X, \cF, \Delta)_\varepsilon}$-MMP for any $\varepsilon>0$.
	
	We start this section with the proof of the existence of  the $K_\cF+\Delta^{n-inv}$-MMP with mild assumption on the singularities of the foliation. This is going to be the main ingredient of the proof of Theorem \ref{th: adjoint doppia foliazione}.
	\subsection{Cone theorem and contraction theorem for foliation on surfaces}
    \begin{theorem}
    \label{conetheoremsurfaces}
        Let $X$ be a normal projective surface, $\cF$ a 
		rank one foliation on $X$,
		and $\Delta = \sum a_iD_i$ an effective $\mathbb R$-divisor.
		Suppose that $K_{\cF}+\Delta$ is $\mR$-Cartier.
		
		Then
		$$\overline{NE}(X) = \overline{NE}(X)_{K_{\cF}+\Delta \geq 0} +
		Z_{-\infty} (\cF,\Delta)+ \sum \mR^+[L_i]$$
		where $L_i$ are invariant rational curves with $(K_{\cF}+\Delta)\cdot L_i \geq -4$, 
		and $Z_{-\infty}$ is spanned by the classes
		$[D_i]$ where $a_i > \epsilon(D_i)$.
		In particular, if $H$ is ample, there are only finitely many curves with
		extremal rays $R$ with $(K_{\cF}+\Delta+H)\cdot R<0$.
        
        Suppose moreover that $X$ is $\mQ$-factorial and $\Delta$ is a boundary with $\Delta=\Delta^{n-inv}\geq 0$. 
        Suppose that $a(\cF,E)\leq a((X,L_i), E)$ for every $E$ divisor over $X$ with $\iota(E)=1$.
        Let $R=\mR_{\geq 0}[\xi]\subset\overline{\cone}(X)$ be a $K_{\mathcal F}+\Delta$-negative extremal ray.
        Then if:
        \begin{enumerate}
            \item $\xi^2\geq 0$, then 
        $\varphi_R$ is a contraction of fibre type, and is a $K_X$-negative contraction and the relative Picard number of the contraction is 1.	
        \item $\xi^2< 0$ then $R$ is spanned by a $\cF$-invariant curve $\xi$ and there exists a birational contraction $\varphi_R:X\to Y$, that contracts only $\xi$. Furthermore, $Y$ is projective, normal, $\mQ$-factorial, and the relative Picard number of the contraction is 1.
        \end{enumerate}
    \end{theorem}

	\begin{remark}
		\label{rm: Z vuoto 2}
		Suppose that $\Delta$ is a boundary, i.e. $\Delta\in [0,1]$, and $\Delta=\Delta^{n-inv}. $ Then $Z_{-\infty}(\cF,\Delta)=\varnothing.$ Indeed, $Z_{-\infty}(\cF,\Delta)$ is spanned $D_i$ in $\text{supp}(\Delta)$ with $a_i > \epsilon(D_i)$.  As a consequence,  if $\Delta=0$, then $Z_{-\infty}(\cF,\Delta)=\varnothing$. Hence if there exists a ray $R$ such that $K_\cF\cdot R<0$, then $R$ is spanned by a rational $\cF$-invariant curve.
	\end{remark}
    \begin{proof}
        The first part of the statement is \cite[Theorem 6.3]{HD}.
        We are left to prove the existence of the contractions. Since $\Delta=\Delta^{n-inv}$, by the previous part of this theorem, 
        $R$ is spanned by an $\cF$-invariant curve $\xi.$ 
        We will distinguish two different cases: $\xi^2<0$ or $\xi^2\geq 0$.

        \textbf{Case 1}: $\xi^2\geq 0$. Since $\xi^2\geq 0$, then $\Delta\cdot \xi\geq 0.$ Hence $K_\cF\cdot \xi<0$. The proof of the existence of the contraction of $R$ follows \cite[Proof Theorem 3.2.(2)]{effective}. For the reader convenience we give a sketch of proof.

        Let $\varphi:X'\to X$ be an F-dlt resolution \cite[Theorem 1.4]{corank}, let $\cG=\varphi^{*}\cF$, and $\Delta'=\varphi^{-1}_*\Delta$. Note that $\cG$, is induced by a fibration, $p:X'\to Y$.
        Let $C$ be a general curve in the family of $\xi$. Then $C'=\varphi^{-1}_*C$ is a fiber of $p$, and  $0>K_\cG\cdot C'=K_{X'}\cdot C'$.
        Since $C'$ is general,  then $\varphi(C')\cdot K_{X'}<0$. 
        Moreover $\varphi(C')$ spans $R$.
        Therefore we can construct the contraction of a $K_\cF$-negative ray of fibre type as a $K_X$-negative extremal contraction and conclude by \cite[Theorem 1.1]{Fujino2010MinimalModelLogSurfaces}, that we can contract $R$. 
        
        \textbf{Case 2:} $\xi^2<0$. Note that if the contraction of $R$ exists, then it contracts only $\xi$. Since $\Delta^{n-inv}$ is made just by non-invariant divisors, then $\Delta^{n-inv}\cdot \xi\geq 0$, and $K_\cF\cdot \xi<0$.
		If  $\cF$ is log canonical around $\xi$ we can conclude by \cite[Theorem 2.8]{effective}. Suppose that $\cF$ is worse than log canonical around $\xi$. Then by Lemma \ref{lm: un punto singolare}, then there exists a unique point $P$ where $\xi$ is worse then log canonical. Moreover for every $R\in \xi \setminus \{P\}$ then $\cF$ is terminal at $R$.
		
		Consider an F-dlt modification of $\cF$, $\pi: Y\to X$ (that exists by \cite[Theorem 1.4]{corank}). Then
		\[K_\cG+\widetilde{\Delta}^{n-inv}+\sum\varepsilon(E_i)E_i+G=\pi^*(K_\cF+\Delta^{n-inv}).\]
		
		Let $E'\coloneqq  \widetilde{\Delta}^{n-inv}+\sum\varepsilon(E_i)E_i+G$, and $Q\coloneqq \widetilde{\xi}\cap E'$. By Lemma \ref{lm: C movable}, then $Q$ is a smooth point for $Y$, and $Q$ is a canonical centre for $\cG$.  Note that $\widetilde{\xi}\cdot E'<1$, otherwise $R\cdot (K_\cF+\Delta^{n-inv})\geq 0$, a contradiction. Since $X$ is $\mQ$-factorial and $a(\cF,E)\leq a((X,C), E)$ for every $E$ divisor over $X$ with $\iota(E)=1$, by Lemma \ref{lm: invariant},  $(K_X+\Delta^{n-inv}+\xi)\cdot \xi<0$.
		By \cite[Theorem 1.1]{Fujino2010MinimalModelLogSurfaces}, we can contract $R$. 
    \end{proof}

	\begin{corollary}

		\label{ non-dc}
		Let  $X$ be a normal, $\mQ$-factorial, projective surface, $\cF$ be a rank one foliation, $\Delta\geq 0$ be a boundary. Suppose moreover that wither $\cF$ is non-dicritical or $(\cF,\Delta)$ is log canonical. 
        Then we may run a $K_\cF+\Delta^{n-inv}$-MMP $\rho:X\to X'$. Let $\cG$ be the induced foliation on $X'$ and let $\Delta'\coloneqq\rho_*\Delta$, then the following properties hold:
		\begin{enumerate}
			\item if $K_\cF+\Delta^{n-inv}$ is pseudo-effective, then $K_\cG+\Delta'^{n-inv}$ is nef;
			\item if $K_\cF+\Delta^{n-inv}$ is not pseudo-effective, then $X'$ has a $K_\cG+\Delta'^{n-inv}$-negative fibration $\pi:X'\to Y'$ with $\rho(X'/ Y')=1$ such that $-(K_\cG+\Delta'^{n-inv})$ is $\pi$-ample.
		\end{enumerate}
	\end{corollary}
\begin{proof}
	The proof follows from standard MMP's argument using Theorem \ref{conetheoremsurfaces} (see Remark \ref{rm: Z vuoto 2}).
	\end{proof}

	\subsection{Running an adjoint MMP}
	In this section we will prove that given an adjoint triple we can run a specific MMP.

	\begin{theorem}
		\label{th: adjoint doppia foliazione}
		Fix $\varepsilon>0$.
		Let $(X, \cF,\Delta)$ be a foliated triple on $X$ a normal projective $\mQ$-factorial surface, and $\Delta\geq 0$ is a boundary. Suppose moreover that $(\cF,\Delta^{n-inv})$ is log canonical or $(\cF,\Delta^{n-inv})$ is a foliated pair where $\cF$ is non-dicritical.
        Then there exists an
        $K_{(X,\cF,\Delta)_\varepsilon}$-MMP $\rho:X\to Y$.
        
        Let  $\cF'$ be the induced foliation on $Y$ and $\Delta'\coloneqq \rho^{-1}_*\Delta$, then the following properties hold:
		\begin{enumerate}
			\item if $K_{(X,\cF,\Delta)_\varepsilon}$ is pseudo-effective, then $K_{(Y,\cF',\Delta')_\varepsilon}$ is nef;
			\item if $K_{(X,\cF,\Delta)_\varepsilon}$ is not pseudo-effective, then $Y$ has a contraction of fibre type $\pi:Y\to Y'$ with $\rho(Y/Y')=1$ such that $-K_{(Y,\cF',\Delta')_\varepsilon}$ is $\pi$-ample.
		\end{enumerate}

        Moreover such MMP factors as
	\[\xymatrix{
		X\ar[r]^{f_0}&X_0\ar[r]^{f_1}& X_1\coloneqq Y}
	\]
	where:
	\begin{enumerate}
		\item[(A)] $f_0$ is a partial $K_{(X,\cF,\Delta)_\varepsilon}$-MMP that contracts only $K_{\cF}+\Delta^{n-inv}$-non-positive curves. Denote  $\cF_0\coloneqq f_{0*}\cF$ and $\Delta_0\coloneqq f_{0*}\Delta$;
		\item[(B)]  $f_1$ is a partial $K_{X_0}+\Delta_0$-MMP which only contracts curves which are $K_{X_0}+\Delta_0$ and $K_{(X_0,\cF_0,\Delta_0)_\varepsilon}$-negative and $K_{\cF_0}+\Delta_0^{n-inv}$-positive, where $\cF_1 = f_{{1}_* }\cF_0$ and $\Delta_2={f_2^{}}_{*}\Delta_1$ .
		
	\end{enumerate}
	
	Moreover, the following hold:
	\begin{enumerate}
		
		\item[(3)] if $(X,\Delta)$ is log canonical and $(\cF,\Delta^{n-inv})$ is log canonical, then 
        $(\cF_0,\Delta_0)$ is log canonical and $(X_i,\Delta_i)$ is log canonical
        for $i=0,1$;
		\item[(4)] for $i=0,1$ if $(X,\Delta)$ is $\eta$-lc for some $\eta>0$ and $(\cF,\Delta^{n-inv})$ is log canonical, then $(X_i,\Delta_i )$ is $\eta'$-lc for $\eta'=\frac{\varepsilon\eta}{\varepsilon+1}$;
        \item[(5)] if the foliated triple $(X,\cF, \Delta)$ has $\varepsilon$-adjoint log canonical singularities, then at each step of the $K_{(X,\cF,\Delta)_\varepsilon}$-MMP $(X_i,\cF_i,\Delta_i)$ has $\varepsilon$-adjoint log canonical singularities.
	\end{enumerate}
	\end{theorem}
	For the reader convenience we state an equivalent formulation that uses the notation of \cite{squadrone}.
	\begin{theorem}
		Fix $t\in [0,1]$. Let $(X,\cF,\Delta,t)$ be an adjoint foliated structure where $(\cF,\Delta^{n-inv })$ is log canonical and $X$ is a normal projective surface. Set $K_{(X,\cF,\Delta,t)}\coloneqq t(K_\cF+\Delta)+(1-t)(K_X+\Delta)$.
		Then we may run a $K_{(X,\cF,\Delta,t)}$-MMP
		$\rho:X\to X'$. Let $\cF'$ be the induced foliation on $X'$ and $\Delta'\coloneqq \rho^{}_*\Delta$, then the following properties hold:
		\begin{enumerate}
			\item if $K_{(X,\cF,\Delta,t)}$ is pseudo-effective, then $K_{(X',\cF',\Delta',t)}$ is nef;
			\item if $K_{(X,\cF,\Delta,t)}$ is not pseudo-effective, then $X'$ has a contraction of fibre type $\pi:X'\to Y'$ with $\rho(X'/ Y')=1$ such that $-K_{(X,\cF',\Delta',t)}$ is $\pi$-ample.
		\end{enumerate}
         Moreover such MMP factors as
	\[\xymatrix{
		X\ar[r]^{f_0}&X_0\ar[r]^{f_1}& X_1\coloneqq Y}
	\]
	where:
	\begin{enumerate}
		\item[(A)] $f_0$ is a partial $K_{(X,\cF,\Delta,t)}$-MMP that contracts only curves $K_{\cF}+\Delta^{n-inv}$-non-positive. Denote with $\cF_0\coloneqq f_{0*}\cF$ and $\Delta_0\coloneqq f_{0}\Delta$;
		\item[(B)]  $f_1$ is a partial $K_{X_0}+\Delta_0$-MMP which only contracts curves which are $K_{X_0}+\Delta_0$ and $K_{(X_0,\cF_0,\Delta_0,t)}$-negative and $K_{\cF_0}+\Delta_0^{n-inv}$-positive, where $\cF_1 = f_{{1}_* }\cF_0$ and $\Delta_2={f_2^{}}_{*}\Delta_1$ .
		
	\end{enumerate}
	
	Moreover, the following hold:
	\begin{enumerate}
		
		\item[(3)] if $(X,\Delta)$ is log canonical and $(\cF,\Delta^{n-inv})$ is log canonical, then 
        $(\cF_0,\Delta^{n-inv}_0)$ is log canonical and $(X_i,\Delta_i)$ is log canonical
        for $i=0,1$;
		\item[(4)] for $i=0,1$ if $(X,\Delta)$ is $\eta$-lc for some $\eta>0$ and $(\cF,\Delta^{n-inv})$ is log canonical, then $(X_i,\Delta_i )$ is $\eta'$-lc for $\eta'=(1-t)\eta$;
        \item[(5)] if the foliated triple $(X,\cF, \Delta)$ has $\varepsilon$-adjoint log canonical singularities, then at each step of the $K_{(X,\cF,\Delta,t)}$-MMP $(X_i,\cF_i,\Delta_i)$ has $\varepsilon$-adjoint log canonical singularities.
	\end{enumerate}
       
	\end{theorem}

	\begin{proof}[Proof of Theorem \ref{th: adjoint doppia foliazione}]
    By Lemma \ref{lm: zariski} and by \cite[Lemma 3.39]{kollár_mori_1998}, we can run a $K_{(X,\cF,\Delta)_\varepsilon}$-MMP by first running an $K_\cF+\Delta^{n-inv}$-MMP (which is guaranteed to exist by Corollary \ref{ non-dc}) followed by a $K_{X}+\Delta$-MMP (which is guaranteed to exist by \cite{Fujino2010MinimalModelLogSurfaces}). Hence items (A) and (B) are a direct consequence of Lemma 
    \ref{lm: zariski}.
    
    We now prove item (3). Suppose that $(\cF,\Delta^{n-inv})$ and $(X,\Delta)$ are log canonical.
    Since $f_0$ is $(K_\cF+\Delta^{n-inv})$-non-positive, by
    \cite[Theorem 2.8]{effective} $(\cF_0,\Delta_0^{n-inv})$ and $(X_0,\Delta_0)$ are log canonical. Since $f_1 $ is $K_{X_0}+\Delta_0$-negative, then we can conclude that $(X_1,\Delta_1)$ is log canonical.
    
    Finally, item (4) is Proposition \ref{cor: sing}.
    
    Item (5) follows from the Negativity Lemma, \cite[Lemma 3.39]{kollár_mori_1998}.
	\end{proof}

    \begin{remark}
        In Theorem \ref{th: adjoint doppia foliazione} we showed that when running a $K_{(X, \cF, \Delta)_\varepsilon}$-MMP we can control the singularities of the underlying surface $(X,\Delta)$ during the MMP. 
    \end{remark}
    \subsubsection{Zariski Decomposition}
  The next proposition gives a description of the curve contracted during $K_\cF+\varepsilon K_X$-MMP when $0<\varepsilon<\frac{1}{5}$, and Zariski decomposition of divisors of the form $K_\cF+\varepsilon K_X$ when $0<\varepsilon<\frac{1}{5}$.
\begin{proposition}
\label{prop: struttura}
Fix $0<\varepsilon< \frac{1}{5}$. Let $(X,\cF,\Delta)$ be a foliated triple where $(X,\Delta)$ is log canonical, $\cF$ is a log canonical foliation and $\Delta^{n-inv}=0.$
Let $\varphi \colon (X, \cF,\Delta) \to (Y, \cG,\Theta)$ be a partial $K_{(X, \cF, \Delta)_\varepsilon}$-MMP, where $\cG$ is the induced foliation on $Y$ and $\Theta=\varphi_*\Delta$. Then $\varphi$ satisfies the following:
\[\xymatrix{(X_0,\cF_0,\Delta_0)\ar^{f_0}[r] &(X_1,\cF_1,\Delta_1)\ar^{f_1}[r]&(X_2,\cF_2,\Delta_2)\ar^{f_2}[r]&(X_3 ,\cF_3,\Delta_3)} \]
         where $(X,\cF,\Delta)\coloneqq (X_0,\cF_0,\Delta_0)$, $(X_3,\cF_3,\Delta_3)\coloneqq(Y, \cG,\Theta)$,  $\Delta_i\coloneqq f_{i-1*}\Delta_{i-1}$, and

\begin{enumerate}
    \item $f_0$ is a $K_\cF$-non-positive contraction;
    \item $f_1$ is a $K_{\cF_1}$-positive, $K_{X_1}+\Delta_1$-negative contraction of disjoint transverse curves;
    \item $f_2$ is a $K_{\cF_2}$-trivial and $K_{X_2}+\Delta_2$-negative contraction of $\cF_2$-invariant curves
\end{enumerate}
where $\Delta_i=f_{i-1*}\Delta_{i-1}$.
\end{proposition}

\begin{proof}
Let $\cG$ be the foliation induced on $Y$.
By Proposition~\ref{pr; rimane lc}, $\cG$ is log canonical. Let $E \coloneqq \Exc(\varphi)$. By Lemma~\ref{lm: zariski}, we can contract all curves that are $K_\cF$-non-positive and $K_{(X, \cF,\Delta)_\varepsilon}$-negative. Let $f_0 \colon X \to X_1$ be such a contraction, and let $\cF_1$ be the induced foliation on $X_1$, and $\Delta_1\coloneqq f_{0*}\Delta_0$.

If $E = \Exc(f_0)$, then we can conclude.

Suppose that $E \neq \Exc(f_0)$. By Lemma~\ref{lm: zariski}, we are left to contract only curves that are $K_{\cF_1}$-non-negative and $K_{X_1}+\Delta_1$-negative. Since $E \neq \Exc(f_0)$, there exists at least one curve $C$ such that
\begin{enumerate}
    \item $C$ is transverse;
    \item $C \cdot K_\cF > 0$;
    \item $(K_X+\Delta) \cdot C < 0$.
\end{enumerate}
Let $f_1 \colon X_1 \to X_2$ be the contraction of such curves. By Proposition~\ref{pr; rimane lc}, $\cF_2$ is log canonical; hence, by~\cite[Lemma~2.6]{effective}, $f_1$ contracts transverse curves that are pairwise disjoint.

Denote $\cF_2 \coloneqq {f_1}_* \cF_1$.
Note that $\cF_2$ is log canonical and $(X_2,\Delta_2)$ is lc.

If $E = \Exc(f_1 \circ f_0)$, then we can conclude. Suppose that $\Exc(f_1 \circ f_0) \neq E$; hence $\rho_{X/Y} = \rho_{X/X_2} + r$, where $r$ is a positive integer.

\begin{claim}
All the remaining $r$ elementary contractions satisfy the following: each contraction is associated with a ray $R \subset \overline{\cone}(X_2)$ that is $K_{(X_2, \cF_2,\Delta_2)_\varepsilon}$-negative and $K_{\cF_2}$-trivial.
Moreover, such a contraction is the contraction of a $\cF_2$-invariant curve.
\end{claim}

\begin{proof}
We prove this claim by induction on the number of steps.
Set $Y_0 \coloneqq X_2$, $\cG_0 \coloneqq \cF_2$ and $\Theta_0\coloneqq\Delta_2$.
Since $r \neq 0$, there exists an extremal ray $R_0$ that is $K_{(Y_0, \cG_0,\Theta_0)_\varepsilon}$-negative.
It remains to show that $R_0$ is $K_{\cG_0}$-trivial and that it is spanned by a $\cG_0$-invariant curve.

For the sake of contradiction, suppose that $R_0$ is $K_{\cG_0}$-negative. Then, by Theorem~\ref{conetheoremsurfaces}, $R_0$ is spanned by a $\cG_0$-invariant curve $\Gamma_0$. By construction, $f_2$ is not an isomorphism on $\Gamma_0$. Hence, there exists a point $P \in \Gamma_0$ that is strictly log canonical for $\cG_0$. Since $\cG_0$ is log canonical, by~\cite[Lemma~2.7]{effective}, $\Gamma_0$ moves — a contradiction.
Therefore, $R_0$ is $K_{\cG_0}$-non-negative.
Let $C_0$ be the curve spanning $R_0$. By~\cite[Lemma~2.6]{effective}, $C_0$ is $\cG_0$-invariant.
For the sake of contradiction, suppose that $R_0$ is $K_{\cG_0}$-positive. Let $f \colon X_2 \to Y_1$ be the contraction of $R_0$. Then the induced foliation on $Y$ is not log canonical, contradicting Lemma~\ref{pr; rimane lc}. Hence, $R_0$ is $K_{\cG_0}$-trivial.

Suppose that for the first $r-1$ elementary contractions
$\varphi_{i-1} \colon Y_{i-1} \to Y_i$ are $K_{(Y_{i-1}, \cF_{i-1},\Theta_{i-1})_\varepsilon}$-negative and $K_{\cF_{i-1}}$-non-negative, where $i = 0, \dots, r-1$, with $Y_0 \coloneqq X_2$, $\cG_i$ the foliation induced on $Y_i$, and $\Theta_i\coloneqq\varphi_{i-1*}\Theta_{i-1}$. Moreover, denote by $R_{i-1}$ the extremal ray associated with $\varphi_{i-1}$. Then $R_{i-1}$ is spanned by an invariant curve.

Let $R_{r-1}$ be an extremal ray of $\overline{\cone}(Y_{r-1})$ such that $R_{r-1}$ is $K_{(Y_{r-1}, \cG_{r-1},\Theta_{r-1})_\varepsilon}$-negative. To conclude, we must show that the following hold:
\begin{enumerate}
    \item $K_{Y_{r-1}}+\Theta_{r-1}$-negative;
    \item $K_{\cG_{r-1}}$-trivial;
    \item $R_{r-1}$ is spanned by an invariant curve.
\end{enumerate}
We can conclude as in the previous part of this claim. Indeed, let $C_{r-1}$ be a curve spanning $R_{r-1}$. By the construction of $f_1$, $\cG_{r-1}$ has a log canonical centre on $C_{r-1}$, and $\cG_{r-1}$ is log canonical.
As before, by~\cite[Lemma~2.6]{effective}, $C_{r-1}$ is $\cG_{r-1}$-invariant. Indeed, over a strictly log canonical point, there exists a unique transverse curve. By~\cite[Lemma~2.7]{effective}, $R_{r-1}$ is $K_{\cG_{r-1}}$-non-negative and $K_{Y_{r-1}}+\Theta_{r-1}$-negative.
If $R_{r-1}$ is $K_{\cG_{r-1}}$-trivial, we can conclude. For the sake of contradiction, suppose that $R_{r-1}$ is $K_{\cG_{r-1}}$-positive. As before, we obtain a contradiction with Lemma~\ref{pr; rimane lc}.
\end{proof}

This concludes the proof.
\end{proof}

\section{Effective Generation}
	
\subsection{Preliminary Results}

We summarize some preliminary results that will be used in the proof of Theorem~\ref{co: bound pseff}.

\begin{lemma}\label{rem: smooth fib}
Fix $0 \leq \tau \leq \frac{1}{2}$.  
Let $(X, \cF, \Delta)$ be a foliated triple where $X$ is a smooth projective variety, $\Delta^{n\text{-inv}} = 0$, and $K_\cF$ is big.  
Suppose that there exists an extremal ray $R \subset \overline{\cone}(X)$ which is $K_{(X, \cF, \Delta)_\tau}$-negative.  
Then $R^2 < 0$.
\end{lemma}

\begin{proof}
Suppose, for the sake of contradiction, that $R$ is a ray with $R^2 \geq 0$.  
Since $K_\cF + \Delta^{n\text{-inv}}$ is big, we have $(K_\cF + \Delta^{n\text{-inv}}) \cdot R > 0$.  
Moreover, by adjunction on $K_X + \Delta$, we obtain $(K_X + \Delta) \cdot C = -2$.  
Therefore, we can conclude that $R$ is not positive on $K_\cF + \Delta^{n\text{-inv}} + \tau (K_X + \Delta)$, a contradiction.  
Hence, the contraction associated with that ray is birational.
\end{proof}

If $(X, \cF, \Delta)$ is not a smooth foliated triple, to obtain similar results we need to bound the Picard number of $X$ from below.

\begin{lemma}\label{biratiobnalcontraction_boundary klt}
Let $(X, \cF, \Delta)$ be a foliated triple where $X$ is a normal $\mQ$-factorial projective surface, $\Delta^{n\text{-inv}} = 0$, $K_\cF$ is big, and $\rho_X \geq 2$.  
Suppose that there exists an extremal ray $R \subset \overline{NE}(X)$ which is $K_{(X, \cF, \Delta)_\tau}$-negative with $\tau \leq \frac{1}{2}$.  
Then $R^2 < 0$.  
Hence, the contraction associated with that ray is birational.
\end{lemma}

\begin{proof}
By contradiction, suppose that there exists a ray $R$ with $R^2 \geq 0$ such that $K_{(X, \cF, \Delta)_\tau} \cdot R < 0$.  
Since $\rho_X \geq 2$, it follows that $R^2 = 0$.  
Let $\varphi \colon X \to C$ be the contraction associated to $R$ and note that $\dim C = 1$.  
Consider a generic fibre $F$ of $\varphi$; then $K_\cF \cdot F \geq 1$ and $(K_X + \Delta) \cdot F \leq -2$, where the first inequality holds because $F$ is the generic fibre and the second inequality holds by adjunction.  
Combining these two inequalities gives
\[
K_{(X, \cF, \Delta)_\tau} \cdot F 
= K_\cF \cdot F + \tau (K_X + \Delta) \cdot F 
\geq 0,
\]
a contradiction.
\end{proof}

We will use the following bound on the Cartier index of a Fano surface.

\begin{lemma}\cite{AM}\label{lm: index}
Fix $\eta \in \mathbb{R}_{>0}$.  
Let $X$ be a normal projective surface and let $\Delta \ge 0$ be such that $(X, \Delta)$ is $\eta$-lc and $-(K_X)$ is nef and big.  
Then the following holds:
\begin{enumerate}
    \item $\rho_{X_{\text{min}}} \leq \frac{128}{\eta^5}$.
    \item Let $D$ be a Weil divisor on $X$. Then $ND$ is Cartier, where 
    \[
    N = \left\lfloor 2 \cdot \left( \frac{2}{\eta} \right)^{\frac{128}{\eta^5}} \right\rfloor!.
    \]
\end{enumerate}
\end{lemma}

The proof of this lemma is contained in \cite{AM}; see also \cite[after Theorem~A]{lai12}.  
For the reader's convenience, we present a sketch of the proof here.

\begin{proof}
Since $X$ is $\eta$-lc, it is klt.  
As klt surface singularities are quotient singularities, it follows that $X$ is $\mQ$-factorial.  
Consider a minimal resolution of singularities $p \colon X_{\rm min} \to X$, and let $(E_{ij})$ be the intersection matrix of $\Exc(p)$.  
Therefore, $(E_{ij})$ is a square matrix of dimension at most $\rho_{X_{\text{min}}} \leq \frac{128}{\eta^5}$ by \cite[Theorem~1.3]{AM}.  
By \cite[Theorem~4.6]{bua}, the order of the divisor class group divides $\det(E_{ij})$.  
Consider an exceptional curve $E$ on $X_{\text{min}}$ over $X$; then 
\[
1 \leq -E^2 \leq \frac{2}{\eta}
\]
by \cite[Lemma~1.2]{AM}.  
Hence,
\[
\det(E_{ij}) \leq 2 \cdot \left( \frac{2}{\eta} \right)^{\rho_{X_{\text{min}}}} 
\leq 2 \cdot \left( \frac{2}{\eta} \right)^{\frac{128}{\eta^5}}.
\]
\end{proof}

\subsection{Behaviour of the pseudo-effective threshold}
       
	 Let $(X,\cF,\Delta)$ be a foliated triple where $X$ is a projective surface, $\cF$ is a rank one foliation, and $\Delta \geq 0$ is a $\mQ$-boundary. Assume that $K_{\cF}+\Delta^{n\text{-inv}}$ is big. 

We want to study the behaviour of the pseudo-effective threshold,
\begin{equation}
    \label{pseff}
    \tau(X,\cF,\Delta) \coloneqq \sup\{t \in \mR \mid K_{(X,\cF,\Delta)_t} \text{ is pseudo-effective}\}.
\end{equation}
If there is no possibility of confusion, we simply denote this threshold by $\tau$. We aim to find an effective upper bound for $\tau(X,\cF,\Delta)$.

First, we will prove that $\tau(X,\cF,\Delta)$ belongs to a DCC set (cf. Theorem \ref{lm: DCC}). Secondly, we will find an effective lower bound for $\tau(X,\cF, \Delta)$ if $\Delta^{n-inv}=0$ (cf. Theorem \ref{bound lo abbiamo}).

\begin{example}
\label{caso P^2}
Let $\cF$ be a foliation on $\mP^2$. By \cite[Chapter 2, Section 3.3]{brunella2015birational}, $\cF$ is determined by its degree $d(\cF)$, i.e., the number of tangency points of $\cF$ with a line $L \subset \mP^2$ that is not $\cF$-invariant:
\[
d(\cF) \coloneqq \tan(\cF, L).
\]
Furthermore, $K_\cF = (d(\cF)-1)H$, where $H$ is a general line in $\mP^2$. Therefore, $K_\cF$ is nef and big if and only if $d(\cF) \geq 2$.

Since $H$ spans the movable cone, the pseudo-effective threshold is
\[
\tau = \frac{d(\cF)-1}{3}.
\]
Hence, the adjoint divisor $K_\cF + \varepsilon K_X$ is pseudo-effective for every $\varepsilon \leq \frac{d(\cF)-1}{3}$.
\end{example}

\subsubsection*{DCC property of the set of pseudo-effective thresholds}

Fix $I \subseteq [0,1]$ and $\varepsilon>0$. Consider the set of all pseudo-effective thresholds $\tau(X,\Delta,\cF)$:

\begin{multline}
\label{set}
\cR_{2,\eta,I,\varepsilon} \coloneqq \{\tau(X,\cF,\Delta) \mid \dim X=2, \\(X,\cF,\Delta) \text{ is a foliated triple with }
\varepsilon\text{-adjoint log canonical singularities}, \\
K_\cF + \Delta^{n\text{-inv}} \text{ big}, (X,\Delta) \text{ is } \eta\text{-lc}, \Delta \in I \cap [0,1]\}.
\end{multline}
In this section, we investigate some properties of $\cR_{2,\eta,I}$.

\begin{theorem}
\label{lm: DCC}
Fix a real number $\eta > 0$. Then:
\begin{enumerate}
    \item If $I \subset [0,1]$ is a finite set, then $\cR_{2, \eta, I,\varepsilon} \cap (0, \delta)$ is finite for any $\delta > 0$;
    \item If $I \subset [0,1]$ is a DCC set, then $\cR_{2, \eta, I,\varepsilon}$ satisfies the DCC.
\end{enumerate}
\end{theorem}

\begin{proof}
Let $(X,\cF,\Delta)$ be a pair as in the definition of $\cR_{2,\eta,I\varepsilon}$. Since $(X,\Delta)$ is $\eta$-lc, $\Delta \in [0, 1-\eta]$. Let $\tau \coloneqq \tau(X,\cF,\Delta)$. By Theorem \ref{th: adjoint doppia foliazione}, we can run a $K_{(X,\cF,\Delta)_\tau}$-MMP
\[
\phi: X \longrightarrow Y,
\]
where $Y$ is a projective surface with $\eta'$-lc singularities for some $\eta'>0$. Note that
\[
K_\cG + \Delta'^{n\text{-inv}} + \tau(K_Y + \Delta') \sim_{\mR} K_Y + \Delta' + \frac{1}{\tau}(K_\cG + \Delta'^{n\text{-inv}}),
\]
where $\Delta' \coloneqq \phi_* \Delta$.

By the definition of $\tau$, $K_Y + \Delta' + \frac{1}{\tau}(K_\cG + \Delta'^{n\text{-inv}})$ is not big. Moreover, it is semi-ample and defines a contraction $f: Y \to Z$. Let $F$ be the general fibre of $f$, then $\dim F > 0$. By definition of $\tau$, restricting to $F$ gives
\[
K_F + \Delta'|_F + \frac{1}{\tau} (K_\cG + \Delta'^{n\text{-inv}})|_F \equiv_\mR 0.
\]
By adjunction, $(F, \Delta'|_F)$ is $\eta'$-log canonical. In particular, $F$ belongs to a bounded family. Hence, we may find a very ample Cartier divisor $A$ on $F$ such that $-K_F \cdot A^{\dim F - 1}$ is bounded. Moreover,
\[
d \coloneqq K_\cG|_F \cdot A^{\dim F - 1} = (\phi_* K_\cF)|_F \cdot A^{\dim F - 1}
\]
is a bounded integer. Intersecting with $A^{\dim F - 1}$, we obtain
\[
c = \sum a_j b_j + \frac{1}{\tau} \left( d + \sum a_j b_j \right), \quad
\frac{1}{\tau} = \frac{1}{d + \sum a_j b_j} \cdot \left( c - \sum a_j b_j \right),
\]
where $b_j = \Delta'_j|_F \cdot A^{\dim F - 1}$ are positive integers, and $c \coloneqq -K_F \cdot A^{\dim F - 1}$ is bounded.

For (1), note that $I$ is finite and $0 < \tau \leq \delta$. Hence, there are only finitely many possibilities for $\frac{1}{\tau}$, and therefore for $\tau$.

For (2), to show that $\tau$ belongs to a DCC set, it suffices to show that $\frac{1}{\tau}$ belongs to an ACC set. The right-hand side of the equation satisfies the ACC, so $\frac{1}{\tau}$ does as well.
\end{proof}

\begin{remark}
If $\Delta$ has coefficients in $\mQ$, then the pseudo-effective threshold $\tau(X,\cF,\Delta) \in \mQ$.
\end{remark}

\begin{remark}
Lemma \ref{lm: DCC} allows us to recover \cite[Theorem 4.1]{effective}. Given a smooth foliated triple $(X,\cF,\Delta)$ where $\Delta \in I$ and $I \subset [0,1]$ is a DCC set, by \cite[Proposition 4.6]{effective} we can find a finite subset $J \subset [0,1) \cap I$ such that the following holds:

Let $(X,\cF,\Delta)$ be a smooth foliated triple where $K_\cF + \Delta^{n\text{-inv}}$ is big and $\Delta \in I$. Then there exists a boundary $0 \leq \Delta' \leq \Delta$ with $\Delta' \in J$ such that $K_\cF + \Delta'$ is big. Applying Lemma \ref{lm: DCC} concludes the argument.
\end{remark}
        \subsubsection{Numerical bound for $\tau$}
	
	In this section we will prove an effective lower bound for $\tau(X,\cF)$, cf. Theorem \ref{bound lo abbiamo}.

		\begin{lemma}
			\label{lm: cartier idex mmmp}
            Fix $\varepsilon>0.$
			Let $(X,\cF,\Delta)$ be a foliated triple where $(X,\Delta)$ is a $\eta$-lc pair, $X$ is a normal projective surface, $\cF$ a rank one foliation on $X$. Consider a $K_{(X,\cF,\Delta)_\varepsilon}$-MMP, i.e. $\rho:(X,\cF,\Delta)\to (X',\cF',\Delta')$.
		Then $({X'},\Delta')$ has $\frac{\eta\varepsilon}{1+\varepsilon}$-lc singularities.

            Moreover, suppose that   $-K_{X'}$ is nef and big.
           Let $D\subset X'$ be a Weil divisor and let $N_D$ be its Cartier index. Then $N_D$ divides
				\[N \coloneqq \left\lfloor 2 \cdot \left( \frac{2}{\frac{\eta\varepsilon}{1+\varepsilon}} \right)^{\frac{128}{\left(\frac{\eta\varepsilon}{1+\varepsilon}\right)^5}} \right\rfloor!.
                \]
                
		  In particular, $NK_{\mathcal F'}$ is Cartier.
		\end{lemma}
		\begin{proof}
            The statement follows from Theorem \ref{th: adjoint doppia foliazione}.
		\end{proof}

		\begin{proposition} 
			\label{bound lo abbiamo}
			Fix $\eta>0.$
			Let $(X, \cF, \Delta)$ be a foliated triple where $(X,\Delta)$ is $\eta$-lc on and $(\cF,\Delta^{n-inv})$ is a rank one log canonical foliation on $X$ such that $K_\cF$ is big. Then one of the following occurs:
			\begin{enumerate}
				\item if $K_X+\Delta$ is pseudo-effective, then $K_{(X,\cF,\Delta)_\varepsilon}$ is pseudo-effective for every $\varepsilon \in [0,+\infty)$;
				\item  if $K_X+\Delta$ is not pseudo-effective fix $\varepsilon>0$, consider a $K_{(X,\cF,\Delta)_\varepsilon}$-MMP, $\rho:(X,\Delta)\to (X',\Delta')$. Let $\cF'$ be the induced foliation on $X'$, and $\Delta'\coloneqq\varphi_*\Delta$. Then one of the following occurs:
				
				\begin{enumerate}
					\item if $K_{(X,\cF,\Delta)_\varepsilon}$ is pseudo-effective, then $K_{(X',\cF',\Delta')_\varepsilon}$ is nef and $$\tau(X',\cF',\Delta')\geq \varepsilon;$$
					\item  if $K_{(X,\cF,\Delta)_\varepsilon}$ is not pseudo-effective, then $X$ is birational via $\rho$ to a Mori fibre space, $f:X'\to Y$.
                    Suppose moreover that  $\Delta^{n-inv}=0$, then we have the following
					\begin{enumerate}
						\item $Y=C$ is a curve, and  $\tau(X',\cF',\Delta')>\frac{1}{2}$;
						\item $Y=P$ a point. Moreover $X'$ is a Fano surface with $\rho_{X'}=1$ and $(X,\Delta')$ has $\frac{\eta\varepsilon}{1+\varepsilon}$-lc singularities.

					\end{enumerate}
				\end{enumerate}
				
			\end{enumerate}
			
		\end{proposition}
		
		\begin{proof}
			First note that $K_\cF+\Delta^{n-inv}$ is pseudo-effective, therefore we need to consider the following two cases:
			\begin{enumerate}
				\item $K_X+\Delta$ pseudo-effective;
				\item $K_X+\Delta$ not pseudo-effective.
			\end{enumerate}
			If $K_X+\Delta$ is pseudo-effective, then 
			$K_{(X,\cF,\Delta)_\varepsilon}$ is pseudo-effective for every $\varepsilon\geq 0.$ Therefore $\tau=+\infty$ and we get case (1) of the statement.

			Suppose now that $K_X+\Delta$ is not pseudo-effective, fix $\varepsilon>0$ and consider a $K_{(X,\cF,\Delta)_\varepsilon}$-MMP with that exists by Theorem \ref{th: adjoint doppia foliazione}:
			\[
			\xymatrix{X	\ar[rr]^{\phi}\ar[rrd]_{f}& 		&	X' \ar[d]^{f'}\\
				&	&Y.
			}	
			\] 
			Denote by $\cF'$ the foliation induced on $X'$ and with $\Delta'\coloneqq\phi_*\Delta$.
			By hypothesis $K_X+\Delta$ is not pseudo-effective and $K_\cF+\Delta^{n-inv}$ is big. If $K_{(X,\cF,\Delta)_\varepsilon}$ is pseudo-effective, then $K_{(X',\cF',\Delta')_\varepsilon}$ is nef. We obtain case (2).(a) of the statement.
			If $K_{(X,\cF,\Delta)_\varepsilon}$ is not pseudo-effective then the MMP will end with a Mori contraction of fibre type $f'$ over $Y$ with $\dim Y\leq 1$ where the fibration $f'$ is $K_{\cF'}$-non negative.
            
			We have to consider the following two cases:
			\begin{enumerate}
				\item $\dim Y=1;$
				\item $\dim Y=0.$
			\end{enumerate}
			Suppose that $\dim Y=1$, i.e. the adjoint MMP ends with a $K_X+\Delta$-negative contraction of fibre type over a curve $C.$ Let $R$ be the ray associated to $f'$. Since $R^2\geq 0$ and $K_{(X',\cF',\Delta')_\varepsilon}$ is not pseudo-effective, then by Lemma \ref{biratiobnalcontraction_boundary klt} we obtain that $\varepsilon>\frac{1}{2}$. We obtain 2.(b).(i) of the statement.

			Suppose now that $\dim Y=0.$ Then $\rho_{X'}=1$, and $X'$ is a rank one log Fano surface, i.e. $K_{X'}$ is anti-ample. 
			Note that by Lemma \ref{lm: cartier idex mmmp} $(X',\Delta')$ has $\eta'$-lc singularities with $\eta'=\frac{\eta\varepsilon}{1+\varepsilon}.$ 
		
		\end{proof}

		\begin{lemma}
			\label{indietro termianl}
			Fix $0\leq \varepsilon $.
			Let $(X,\cF,\Delta)$ be a foliated triple with $K_\cF+\Delta^{n-inv}$ big and $(X,\Delta)$ a log canonical pair. 
		 	Let $f:X\to Y$ is a partial $K_{(X,\cF,\Delta)_\varepsilon}$-MMP which contracts only curves that are $K_\cF+\Delta^{n-inv}$-negative. Let $\cG$ be the induced foliation on $Y$, and $\Delta'=f_*\Delta.$  
		 	If $K_{(Y,\cG,\Delta')_\beta}$ is pseudo-effective for some $\beta<\varepsilon$, then $K_{(X,\cF,\Delta)_\beta}$ is pseudo-effective.
		\end{lemma}
		\begin{proof}
			Since $\varepsilon>\beta$ and $f$ is $K_\cF+\Delta^{n-inv}$-non-positive, we can conclude by the Negativity Lemma,  \cite[Lemma 3.39]{kollár_mori_1998}.
		\end{proof}

	\begin{theorem}
		\label{co: bound pseff}
        Let $(X,\cF,\Delta)$ be a foliated triple where $(X,\Delta)$ is $\eta$-lc for some $\eta>0$, $\cF$ is a log canonical foliation with $K_\cF$ is big, and $\Delta^{n-inv}=0.$
		Then the pseudo-effective threshold  satisfies
		
				$$ \tau(X,\cF,\Delta)\geq  \tau'(\eta)\coloneqq \frac{1}{3}\cdot \frac{1}{\left\lfloor2\left( \frac{2}{\frac{\eta\tau}{1+\tau}} \right)^{ \frac{(2)^7}{\left( {\frac{\eta\tau}{1+\tau}} \right)^5}  }\right\rfloor! }~$$

                where \[\tau\coloneqq\frac{1}{3} \cdot\frac{1}{\left\lfloor 2 \cdot \left( \frac{14}{\eta} \right)^{\frac{128}{\left(\frac{\eta}{7}\right)^5}} \right\rfloor!}~.
\]

	\end{theorem}

    \begin{remark}
        A similar result for nef and big divisor has been proven recently by \cite[Theorem 1.4]{bie2025explicit}. If we consider a foliation $\cF$ on a smooth surface $X$ with $K_\cF$ nef and big, we obtain that $\tau(X,\cF)\geq \frac{1}{3}$.
    \end{remark}
    
    \begin{proof}
    
        Set $\tau \coloneqq \frac{1}{3} \cdot\frac{1}{\left\lfloor 2 \cdot \left( \frac{14}{\eta} \right)^{\frac{128}{\left(\frac{\eta}{7}\right)^5}} \right\rfloor!}$ and  $\tau'\coloneqq \frac{1}{3}\cdot \frac{1}{\left\lfloor2\left( \frac{2}{\frac{\eta\tau}{1+\tau}} \right)^{ \frac{(2)^7}{\left( {\frac{\eta\tau}{1+\tau}} \right)^5}  }\right\rfloor! }~$~.

		Consider the divisor $K_\cF+ \frac{1}{6}(K_X+\Delta)$. If  $K_\cF+ \frac{1}{6}(K_X+\Delta)$ is pseudo-effective, then we are done, indeed $
        \frac{1}{6}\leq \tau(X,\cF,\Delta)$. 
		Suppose that $K_\cF+ \frac{1}{6}(K_X+\Delta)$ is not pseudo-effective. By Theorem \ref{th: adjoint doppia foliazione}, we can run a $K_\cF+ \frac{1}{6}(K_X+\Delta)$-MMP, $\rho:X\to X'$.  Denote
        with $\cF'$ the induce foliation on $X'$, and  $\Delta'\coloneqq\rho_{*}\Delta$.
       By Proposition \ref{bound lo abbiamo}, $X'$ is a Mori fibre space; denote with $f:X'\to Z$ the Mori fibre space structure of $X'$.
       If $\dim Z=1$, by Proposition  \ref{bound lo abbiamo} we conclude that $K_{\cF'}+\frac{1}{6} K_{X'}$ is big. 
       By the Negativity Lemma $K_\cF+ \frac{1}{6}(K_X+\Delta)$ is pseudo-effective.

		Suppose that we are in Case 2.b.ii of Proposition \ref{bound lo abbiamo}.  Then $\rho_{X'}=1$, and $X'$ is a Picard rank one Fano surface, i.e. $K_{X'}$ is anti-ample. By Lemma \ref{lm: cartier idex mmmp}.(3), the Cartier index of the foliated canonical divisor $i(\cF')$ divides $ 1\wedge\tau,$
       and $K_{\cF'}+\frac{1}{3}\cdot \tau\cdot(  K_{X'}+\Delta')$
        is pseudo-effective by the Cone theorem for surfaces \cite{kollár_mori_1998}. By Proposition \ref{prop: struttura}, then $\rho$ factors in the following way:
        \[\xymatrix{(X_0,\cF_0,\Delta_0)\ar^{f_0}[r] &(X_1,\cF_1,\Delta_1)\ar^{f_1}[r]&(X_2,\cF_2,\Delta_2)\ar^{f_2}[r]&(X_3 ,\cF_3,\Delta_3)} \]
        where $(X,\cF,\Delta)\coloneqq (X_0,\cF_0,\Delta_0)$, $(X_3,\cF_3,\Delta_3)\coloneqq(X',\cF',\Delta')$,  $\Delta_i\coloneqq f_{i-1*}\Delta_{i-1}$, and
        \begin{enumerate}
 		\item $f_0$ is a $K_\cF$-non-positive contraction;
 		\item $f_1$ is a $K_{\cF_1}$-positive $K_{X_1}+\Delta_1$-negative contraction of disjoint transverse curves;
 		\item $f_2$ is $K_{\cF_2}$-trivial and $K_{X_2}+\Delta_2$-negative contraction of $\cF_2$-invariant curves. 
 	\end{enumerate}
     Since $f_2$ is $K_{\cF_2}$-trivial, $K_{(X_2,\cF_2,\Delta_2)_\tau}$ is pseudo-effective.
        Write 
        \[K_{(X_0,\cF_0,\Delta_0)_\tau}+E_\tau=\rho^{*}K_{(X_2,\cF_2,\Delta_2)_\tau}+F_\tau\]
        where $E_\tau$ and $F_\tau$ are $\rho $-exceptional effective divisors with no common component.
        If $E_\tau=0$, we can conclude.
        Suppose that $E_\tau\neq0$, then there exists a point $P\in X_2$ that is not $\tau$-adjoint lc. Therefore  $P$ is strictly log canonical at $P$ and $f_1$ is contracting a transverse divisor over $P$. Let $\pi:\widetilde{X}\to X_2$ be a partial resolution guaranteed by Lemma \ref{lem_resolution_logsmooth_near_lc},  let $\widetilde{\cF}$ be the pull-back foliation on $\widetilde{X}$, and let $\widetilde{\Delta}\coloneqq\pi^{-1}_*\Delta_2$. Then there exists a map $\widetilde{\rho}:X\to \widetilde{X}$ and $K_{(X,\cF,\Delta)_\beta}=\widetilde{\rho}^*K_{(\widetilde{X},\widetilde{\cF},\widetilde{\Delta})_\beta}+E_\beta$  where $E\geq 0$ is a $\widetilde{\rho}$-exceptional divisor for every $\beta<\frac{1}{6}.$ Therefore, if one proves that $K_{(\widetilde{X},\widetilde{\cF},\widetilde{\Delta})_\beta}$ is pseudo-effective for some $\beta$, than also $K_{(X,\cF,\Delta)_\beta}$ is pseudo-effective.
        If $K_{(\widetilde{X},\widetilde{\cF},\widetilde{\Delta})_\tau}$ is pseudo-effective we can conclude. 
        Otherwise consider a $K_{(\widetilde{X},\widetilde{\cF},\widetilde{\Delta})_\tau} $-MMP: $g:\widetilde{X}\to Y$, let $\cG$ be the induced foliation on $Y$ and $\Theta\coloneqq g_{*}\widetilde{\Delta}$. By Proposition \ref{bound lo abbiamo}, $Y$ is a Fano variety with $\frac{\eta\tau}{1+\tau}$-lc singularities.  
        \[\xymatrix{ & & \widetilde{X}\ar^{\pi}[d]\ar^{g}[r] & Y \\
        X\ar^{\widetilde{\rho}}[rur]\ar_{f_0}[r] & X_1\ar_{f_1}[r]& X_2 \ar_{f_2}[r]& X'
        }\]
        \begin{claim}
            $g$ is $K_{\widetilde{\cF}}$-non positive.
        \end{claim}
        \begin{proof}
        For sake of contradiction suppose that $g$ is contracting a curve $K_{\widetilde{\cF}}$-positive curve $C.$
            Denote by $h\coloneqq g\circ\widetilde{\rho}:X\to Y$. Note that $\widetilde{\rho}$ is a $K_{(X,\cF,\Delta)_\tau}$-negative and $K_{(X,\cF,\Delta)_\frac{1}{6}}$-negative.
             Since $\widetilde{\cF}$ is log canonical by Proposition \ref{pr; rimane lc}, then $C$ is non-invariant. Let $C'$ be the strict transform of $C $ in $X.$ 
            To conclude, it is enough to show that $C'\subset N_\tau\cap N_\frac{1}{6}$ where $K_\cF + \tau (K_X +\Delta)+h^*H=N_\tau+P_\tau$  and $K_\cF + \frac{1}{6} {(K_X+\Delta)} +h^*H=N_\frac{1}{6}+P_\frac{1}{6}$ where are the Zariski decomposition of $K_\cF + \tau (K_X+\Delta)$ and $K_\cF + \frac{1}{6} (K_X+\Delta)$ respectively, where $H$ is sufficiently ample divisor on $Y.$
            By Lemma \ref{lm: zariski}, then we can factor $h$ in the following way: $h=h_1\circ h_0$  where $h_0$ is $K_X+\Delta$-non positive and $h_1$ is $K_{\cF_1}$-negative and $K_{X_1}+\Delta_1$-non-negative. Since $\tau<\frac{1}{6}<\frac{1}{5}$, then the induced foliation by $f_0$ is still log canonical.
            
            Since $\tau<\frac{1}{6}$, then $C'\subset N_\tau\cap N_\frac{1}{6}$. By the rigidity Lemma \cite[Lemma 1.12]{debarre}, $Y=X'$ a contradiction. Hence $g$ is $K_\cF$-non-positive.
        \end{proof}
    Then we can conclude that $K_{(Y,\cG,\Theta)_{\tau'}}$ is pseudo-effective and $K_{(X,\cF,\Delta)_{\tau'}}$ is pseudo-effective. 
    \end{proof}

    \begin{corollary}
    Let $\cF$ be a log canonical foliation on a surface $X$ with canonical singularities.
        Then the pseudo-effective threshold  satisfies
		
				$$ \tau(X,\cF)\geq  \tau'\coloneqq \frac{1}{3}\cdot \frac{1}{\left\lfloor2\left (\frac{2}{\tau}+2\right )^{ \frac{128}{\left (\frac{\tau}{1+\tau}\right )^5}  }\right\rfloor! }~$$

                where 
                where \[\tau\coloneqq\frac{1}{3} \cdot\frac{1}{\left\lfloor 2 \cdot \left( {14} \right)^{{128\cdot 7^5}} \right\rfloor!}~.
\]
    \end{corollary}
    \begin{proof}
        The claim follows form Theorem \ref{co: bound pseff} with $\varepsilon=\frac{1}{6}$ and $\eta=1$.
    \end{proof}
 \begin{remark}
    Let
    \[
        N=\big\lfloor 2\cdot 14^{2151296}\big\rfloor = 2\cdot 14^{2151296},
    \]
    and
    \[
        \tau=\frac{1}{3\cdot N!}.
    \]

    Then
    \[
        \alpha \;=\; \frac{\tau}{1+\tau} \;=\; \frac{1}{3N!+1},
    \]
    and therefore
    \[
        \tau'=\frac{1}{3}\cdot 
        \frac{1}{\Big\lfloor 2\Big(\frac{2}{\alpha}\Big)^{\frac{128}{\alpha^5}}\Big\rfloor!}
        \;=\; \frac{1}{3\Big(2\big(2(3N!+1)\big)^{128(3N!+1)^5}\Big)!}.
    \]

    Substituting \(N=2\cdot 14^{2151296}\) yields the following expression:
    \[
        \tau' \;=\; 
        \frac{1}{3\;\Big(2\big(2\big(3(2\cdot 14^{2151296})!+1\big)\big)^{128\big(3(2\cdot 14^{2151296})!+1\big)^5}\Big)! }.
    \]
\end{remark}

\subsection{Explicit effective birationality}
We are left to find an explicit value for $M$ such that $M(K_\cF+\tau K_X)$ defines a birational map.
The existence of a universal $M$ is guaranteed by \cite[Corollary 4.8]{effective}.
We will now give an explicit value for $M$.
We will apply some results from \cite{acclc} and from \cite{MR4588595}.
\subsubsection{A lower bound on the volume guarantees effective birationality}
In this section we show that finding a lower bound on the volume of the adjoint divisor allows us to recover an explicit result for effective birationality.

\begin{lemma}\label{rm: bound}
Fix $\eta>0$.
Let $\cF$ be a foliation on $X$, a projective surface with $\eta$-lc singularities. Suppose that $K_\cF+\tau K_X$ is big.
Then $|m(K_\cF+\tau K_X)|$ defines a birational map for every
\[
m\geq \left(\frac{1}{\vol\left(\frac{1}{\tau}K_\cF+K_X\right)}\cdot v(\eta)\right)^{\tfrac{1}{2}},
\]
where
\[
v(\eta) = 64\left( \max \left\{ \frac{192}{\eta} + 1,~ \left( 45 \left( \frac{2}{\eta}! \right)^2 + 1 \right) \right\} \right)^2.
\]
\end{lemma}

\begin{proof}
Let $m$ be the smallest integer such that $|m(K_\cF+\tau K_X)|$ defines a birational map, and let
\[
v(\eta) = 64\left( \max \left\{ \frac{192}{\eta} + 1,~ \left( 45 \left( \frac{2}{\eta}! \right)^2 + 1 \right) \right\} \right)^2.
\]
By \cite{bie2025explicit} there is a constant $v$ such that either $m\leq v$ or $\vol\left(m{\left(\frac{1}{\tau}K_\cF+K_X\right )}\right)\leq v$. Hence we have the inequality
\[
\vol\left (\frac{1}{\tau}K_\cF+K_X\right )\cdot m^2\leq v.
\]
Therefore
\[
m\leq \left(\frac{1}{\vol\left(\frac{1}{\tau}K_\cF+K_X\right)}\cdot v\right)^{\tfrac{1}{2}}.
\]
\end{proof}

\begin{remark}
Thanks to Lemma \ref{rm: bound}, to bound $m$ it is enough to give a lower bound of $\vol\left(\frac{1}{\tau}K_\cF+K_X\right )$.
\end{remark}

\begin{lemma}\label{lm: general type}
Fix $\eta>0$.
Let $\cF$ be a foliation on $X$, a projective surface with $\eta$-lc singularities. Suppose that $K_\cF$ is big and $\kappa(K_X)=2$. Then 
\[
\left|m(K_\cF+K_X)\right|
\]
defines a birational map for every
\[
m>v(\eta)\cdot 42\cdot 84^{128\cdot 42^5+168},
\]
where
\[
v(\eta) = 64\left( \max \left\{ \frac{192}{\eta} + 1,~ \left( 45 \left( \frac{2}{\eta}! \right)^2 + 1 \right) \right\} \right)^2.
\]
\end{lemma}

\begin{proof}
Since $\kappa(K_X)=2$, $K_X$ is pseudo-effective and $K_X+K_\cF$ is big. 
By Lemma \ref{rm: bound}, if $m\geq \frac{v(\eta)}{\vol(K_\cF+K_X)}$, then $|m(K_\cF+K_X)|$ defines a birational map.
We are left to bound $\vol(K_\cF+K_X)$.
Consider a $K_X$-MMP $\rho:X\to X'$. By the Negativity Lemma \cite[Lemma 3.39]{kollár_mori_1998}, $\vol(K_X)\geq \vol (K_{X'})$. Hence $\vol(K_\cF+K_X)\geq \vol(K_X)\geq \vol(K_{X'})\geq \frac{1}{42\cdot 84^{128\cdot 42^5+168}}$, where the last inequality follows from \cite{AM}.
Then $|m(K_\cF+K_X)|$ defines a birational map for every $m\geq v(\eta)\cdot 42\cdot 84^{128\cdot 42^5+168}$.
\end{proof}

\begin{lemma}\label{lm: fano case} 
Fix $\eta>0$ and $0<\varepsilon_0<\tau$ where $\tau$ is the constant determined in Theorem \ref{bound lo abbiamo}.
Let $\cF$ be a rank one foliation on $X$, a $\mQ$-factorial projective surface with $\eta$-lc singularities. Suppose that $K_\cF$ is big and $X$ is Fano. Then 
\[
|m(K_\cF+\varepsilon_0 K_X)|
\]
defines a birational map for every
\[
m\geq 
8!\left (2+8\cdot \frac{1}{\tau}\cdot \left\lfloor2\left (\frac{2}{\eta}\right )^{ \frac{(2)^7}{ \eta^5}  }\right\rfloor! \right).
\]
\end{lemma}

\begin{proof}
By Lemma \ref{lm: cartier idex mmmp}.(3), the Cartier index of $X'$ divides
\[
i\coloneqq \left\lfloor2\left (\frac{2}{\eta}\right )^{ \frac{(2)^7}{ \eta^5}  }\right\rfloor! ,
\]
and by Theorem \ref{bound lo abbiamo}, $NK_\cF+ K_X$ is big. 
Then $i(NK_\cG+K_Y)$ is a big Cartier divisor.
For $m\gg1$, 
\[
mi(NK_\cG+K_Y)-K_Y
\]
is nef and big. Set $M(m)=8!(2+m)$.
By \cite[Theorem 1.1]{kollar}, $|M(m)i(NK_\cG+K_Y)|$ defines a birational map.
Therefore if we bound $m$, we can conclude. 
Note that $m(NK_\cG+K_Y)-K_Y$ is big for every $m\geq 1$.
Suppose there exists a ray $R\subset\overline{\cone}(X)$ that is $mNK_\cG+(m-1)K_Y$-negative; then $R\cdot K_\cG<0$. Hence $R$ is spanned by a $\cG$-invariant curve $C$ such that $C^2<0$.
Note that $K_Y\cdot C<\frac{mN(-K_\cG\cdot C)}{m-1}\leq8N$.
Therefore 
\[
m<\frac{K_Y\cdot C}{i(NK_\cG+K_Y)}\leq\frac{8N}{1}.
\]
Therefore if we take $m\geq {8Ni}$, we can conclude that $a(NK_\cG+K_Y)$ is big, nef and Cartier. Hence
\[
M=8!(8Ni+2)=8!\left (2+8\cdot \frac{1}{\eta}\left\lfloor2\left (\frac{2}{\eta}\right )^{ \frac{(2)^7}{ \varepsilon^5}  }\right\rfloor! \right)
\]
gives the desired constant.
\end{proof}

\begin{remark}
Analogously to Lemma \ref{lm: general type}, we can use Lemma \ref{rm: bound} to conclude.
\end{remark}

\subsubsection{Viehweg product trick}
In this section, we will use the Viehweg product trick to prove the effective birationality of adjoint linear systems on varieties that admit a fibration onto a curve. 

We will need the following definition.
\begin{definition}
Let $C$ be a smooth curve. A locally free sheaf $M$ on $C$ is called \emph{weakly positive} if every quotient sheaf of
$M$ has non-negative degree. 
\end{definition}

\begin{construction}[Viehweg product trick]\label{prod}
\cite[Section 3.1.22]{itaka} Fix $k\geq 1$. 
Let $f:X\to Y$ be a contraction from a normal quasi-projective variety $X$ to $Y$, a Gorenstein variety. Let $\cF$ be a foliation on $X$.
We denote by $X^k\coloneqq X\times_Y\cdots\times_YX$ the $k$-fold fibre product, denote by $f^k:X^k\to Y$ the morphism to $Y$ and by $\cF^k$ the induced foliation on $X^k$. Let $\pi_s:X^k\to X$ be the projection to the $s$th factor.
\[
\xymatrix{X^k\ar^{\pi_s}[r]\ar_{f^k}[dr]&X\ar^{f}[d]\\
&Y.
}
\]

After possibly restricting to an open subset, we can assume that $X^k$ is Gorenstein.
Denote by $E\coloneqq f_*\cO(nK_{X}+mK_\cF)$ and $E_{/Y}\coloneqq f_*\cO(nK_{X/Y}+mK_\cF)$. Then 
$f_*^k\cO(nK_{X^k/Y}+mK_{\cF^k})\coloneqq E_{/Y}^{\otimes k}$ and $f_*^k\cO(nK_{X^k}+mK_{\cF^k})\coloneqq E^{\otimes k}$.
\end{construction}

\begin{lemma}\label{lm: fibration}
Let $\cF$ be a rank one foliation on $X$, a $\mQ$-factorial projective surface with $\eta$-lc singularities for some $\eta>0$. Suppose that  $K_\cF$ is big and $X$ has a Mori fibre space structure $f:X\to C$, where $C$ is a curve with $g(C)\geq 2$. Then 
\[
\left|m\left(5K_\cF+K_X\right)\right|
\]
defines a birational map for every $m>1$.
\end{lemma}

\begin{proof}
By Theorem \ref{bound lo abbiamo}, $5K_\cF+K_X$ is big and nef.
\begin{claim}
We can replace $X$ by a resolution $\pi:X'\to X$ and $\cF'$ by the pull-back foliation on $X'$.
\end{claim}
\begin{proof}
Since $X$ is klt, it has rational singularities. Let $\pi:X'\to X$ be a resolution of $X,$ then  $\pi_*K_{X'}=K_X$.
By projection formula, $(f\circ \pi)_*\pi^*(5K_\cF+K_X)=f_*(5K_\cF+K_X)$. Hence, we can replace $X$ with $X'.$
\end{proof}

From now on we assume $X$ is smooth. Using the notation of Construction \ref{prod}, let $\mu:Y\to X^k$ be a resolution of $X^k$. Denote by $\cG$ the pull-back foliation on $Y$.  

\begin{claim}
$f_*\cO(10K_\cF+2K_{X/C})\otimes \cO(2P)$ is generically globally generated and nef, where $P$ is a general point on $C$. 
\end{claim}
\begin{proof}
Let $a,b$ be positive intergers such that $aK_\cF+bk_X$ is nef and big.
Consider the following short exact sequence:
\[
\begin{split}
0\to\cO(K_Y+K_\cG+\mu^*((b-1)K_{\cF^k}+(a-1)K_{X^k/C})+X_0)\to\\ 
\cO(K_Y+K_\cG+\mu^*((b-1)K_{\cF^k}+(a-1)K_{X^k/C})+2X_0)\to\\ 
\cO_{X_0}((K_\cF+K_{X/C})_{|_{X_0}})\to0.
\end{split}
\]
By taking the long exact sequence in cohomology,
\[
\begin{split}
H^0(X,\cO(k(bK_\cF+aK_{X/C})+2X_0))\to H^0(F,\cO(k(bK_\cF+aK_{X/C})_{|_{X_0}}))\to \\
H^1(X,\cO(k(bK_\cF+aK_{X/C})+X_0))\to \cdots
\end{split}
\]

By the Kawamata–Viehweg vanishing theorem we have that 
$H^1(X,\cO(k(bK_\cF+aK_{X/C})+X_0))=0$ for every $k\geq 1$, since $k(bK_\cF+aK_{X/C})+X_0$ is nef and big for every $k\geq 1$.
Therefore $ f_*(\cO(bK_\cF+aK_{X/C}))\otimes\cO(2P)$ is generically generated by global sections. By Theorem \ref{bound lo abbiamo}, we can set $a=10$ and $b=5$.
\end{proof}

By the previous claim, $f_*(10K_\cF+2K_{X/C})\otimes \cO(2P)$ is generically globally generated. Since $K_C\geq 1$, we can conclude that $f_*(10K_\cF+2K_X)$ is generically globally generated.  
\end{proof}

\begin{lemma}\label{kodaria 1}
Fix $\eta>0$.
Let $\cF$ be a foliation on $X$, a projective surface with $\eta$-lc singularities. Suppose that $K_\cF$ is big, and $\kappa(K_X)=1$. Then 
\[
\left|m(K_\cF+K_X)\right|
\]
defines a birational map for every $m\geq24$.
\end{lemma}

\begin{proof}
Since $K_X$ is pseudo-effective, $K_X+K_\cF$ is pseudo-effective and big. Consider a $K_{(X,\cF)_1}$-MMP $\rho:X\to X'$. Then $K_{(X',\cF')_1}$ is nef and big, where $\cF'$ is the induced foliation on $X'$, and $X'$ is canonical by \cite{adjoint_negative}.
Since $\kappa(K_X)=1$, there exists a contraction $f:X'\to Z$ where $\dim Z=1$ and $Z$ is smooth.
As in Lemma \ref{lm: fibration}, $f_{*}(2K_{X/C}+2K_\cF)\otimes\cO(2P)$ is generically globally generated and $f_{*}(2K_{X/C}+2K_\cF)$ is nef.
If $g(C)\geq 2$, then we conclude that $f_{*}(2K_{X/C}+2K_\cF)$ is generically globally generated, hence $|2K_{X}+2K_\cF|$ defines a map with two-dimensional image. Hence, for $m\geq 2\cdot 3=6$, $|m(K_\cF+K_X)|$ defines a map with two-dimensional image. 
If $g(C)=0$, then $f_{*}(2K_{X/C}+2K_\cF)$ is nef and splits into a line bundle of degree $d\geq 0$. Also in this case, for every $m\geq 6$, $|m(K_\cF+K_X)|$ defines a birational map.

If $g(C)=1$, let $\mu : Y\to X'$ be a resolution of $X'$. Then $\kappa(K_Y)=1$. Moreover $|12K_Y|$ defines a map onto a curve. Since $X'$ is canonical, $|12K_{X'}|$ defines a map onto a curve. Since $|12K_{X'}|+|12K_{\cF'}|\subset |12(K_{X'}+K_{\cF'})|$, we have $h^0(X',2(K_{X'}+K_{\cF'}))\geq2$, and $|24K_X-K_X|\neq \emptyset.$ By \cite[Corollary 1.7]{Zhu2023Explicit}, $|m(K_\cF+K_X)|$ defines a map with two-dimensional image for $m\geq 24$.
\end{proof}

\begin{lemma}\label{lm: fibration el}
Let $\cF$ be a rank one foliation on $X$, a $\mQ$-factorial projective surface with $\eta$-lc singularities for some $\eta>0$. Suppose that $(X,\cF)$ is $\tfrac{1}{5}$-adjoint canonical, $K_\cF$ is big and $X$ has a Mori fibre space structure $f:X\to C$, where $C$ is a curve with $g(C)= 1$. Then 
\[
\left|m\left(K_\cF+\tfrac{1}{5}K_X\right)\right|
\]
defines a birational map for every $m\geq2$.
\end{lemma}

\begin{proof}
If $-K_X$ is nef and big, then we can conclude as in Lemma \ref{lm: fano case}. From now on we assume that $-K_X$ is nef but not big. 

We treat the following cases separately: $-K_X$ abundant, and $-K_X$ not abundant.

\textbf{Case 1: $-K_X$ abundant.} Since $-K_X$ is nef but not big and $X$ is ruled, $-K_X$ induces a map $\varphi:X\to \mP^1$.
As in Lemma \ref{lm: fibration}, we can assume that $X$ is smooth. By \cite{Viehweg1983WeakPA}, $f_*(5K_\cF+K_X)$ and $f_*(5K_\cF+K_X)\otimes\cO(2P)$ are semi-positive. Since $f_*(5K_\cF+K_X)$ is nef and is a vector bundle on $\mP^1$, it splits into line bundles that are generated by global sections. Hence we can conclude. 

\textbf{Case 2: $-K_X$ not abundant.} 
Since $X$ is not rational, by \cite[Theorem 1.4]{Gongyo2010LogSurface} we have that $X$ is smooth. Hence we can conclude.
\end{proof}

\subsubsection{An explicit bound for effective birationality}

\begin{theorem}\label{lm: bound M}
Fix $\eta >0$ and $0<\varepsilon<\tau_0$ where $\tau_0$ is the constant determined in Theorem \ref{bound lo abbiamo}.
Let $\cF$ be a log canonical foliation on $X$, a $\eta$-lc surface. Suppose that $K_\cF$ is big. Then
\[
|m(K_\cF+\varepsilon K_X)|
\]
defines a birational map for every
\[
m\geq M(\eta,\varepsilon)\coloneqq8!\left (2+8\cdot \frac{1}{\varepsilon}\cdot \left\lfloor2\left (\frac{2}{\frac{\eta\varepsilon}{\varepsilon+1}}\right )^{ \frac{(2)^7}{ \frac{\eta\varepsilon}{\varepsilon+1}^5}  }\right\rfloor! \right )
\]
if $K_X$ is not pseudo-effective,  or
\[
m>M(\eta,\varepsilon)\coloneqq v(\eta)\cdot 42\cdot 84^{128\cdot 42^5+168},
\]
where
\[
v(\eta) = 64\left( \max \left\{ \frac{192}{\eta} + 1,~ \left( 45 \left( \frac{2}{\eta}! \right)^2 + 1 \right) \right\} \right)^2
\] if $K_X$ is pseudo-effective and $\kappa(K_X)\neq0$.
\end{theorem}

\begin{proof}
Suppose $K_X$ is pseudo-effective.
If $\kappa(K_X)=2$, then we conclude by Lemma \ref{lm: general type}, and $|m(K_\cF+K_X)|$ defines a birational map for every 
\[
m>v(\eta)\cdot 42\cdot 84^{128\cdot 42^5+168},
\]
where
\[
v(\eta) = 64\left( \max \left\{ \frac{192}{\eta} + 1,~ \left( 45 \left( \frac{2}{\eta}! \right)^2 + 1 \right) \right\} \right)^2.
\]
Suppose $\kappa(K_X)=1$. Then consider a $K_{(X,\cF)_\varepsilon}$-MMP $\rho:X\to X'$. Since $\kappa(K_X)=1$, we conclude by Lemma \ref{kodaria 1}.
If $\kappa(K_X)=0$, then $X$ is not rational.

Suppose that $K_X$ is not pseudo-effective. 
Consider a $K_{(X,\cF)_\varepsilon}$-MMP $\rho:X\to Y$ and denote by $\cG$ the induced foliation on $Y$. By Theorem \ref{co: bound pseff}, $K_\cF+\varepsilon K_X$ is pseudoeffective and $K_\cG+\varepsilon K_Y$ is nef and big.
We consider the following cases:
\begin{enumerate}
    \item $Y$ is birational via a $K_Y$-MMP to a Mori fibre space $Z\to C$ where $C$ is a smooth projective curve;
    \item $Y$ is a Fano variety with $\rho_Y=1$ and $\frac{\eta\varepsilon}{\varepsilon+1}$-lc singularities;
    \item $Y$ is birational via a $K_Y$-MMP to a Fano variety $Z'$ with $\rho_{Z'}=1$ and $\frac{\eta\varepsilon}{\varepsilon+1}$-lc singularities.
\end{enumerate}
If $Y$ is birational via a $K_Y$-MMP to a Mori fibre space $Z\to C$ where $C$ is a smooth projective curve, then we can conclude by Lemmata \ref{lm: fibration} and \ref{lm: fibration el}.
If $Y$ is a Fano variety with $\rho_Y=1$, then we can conclude by Lemma \ref{lm: fano case} that $|m(K_\cF+\varepsilon K_X)|$ defines a map with two-dimensional image for 
\[
m\geq M(\eta,\varepsilon)\coloneqq\left (2+8\cdot \frac{1}{\varepsilon}\left\lfloor2\left (\frac{2}{\frac{\eta\varepsilon}{\varepsilon+1}}\right )^{ \frac{(2)^7}{ \left(\frac{\eta\varepsilon}{\varepsilon+1}\right)^5}  }\right\rfloor! \right).
\]

If $Y$ is birational via a $K_Y$-MMP to a Fano variety $Z'$ with $\rho_{Z'}=1$ and $\frac{\eta\varepsilon}{\varepsilon+1}$-lc singularities, notice that $\tau\coloneqq\tau(X,\cF)>\varepsilon$, and consider a $K_{(Y,\cG)_\tau}$-MMP $f:Y\to Y'$. By hypothesis $\rho_{Y'}=1$ and $Y'$ is Fano. By construction $f$ is a $K_{Y}$-MMP, so $K_Y=f^*(K_{Y'})+E$ for some effective $f$-exceptional divisor $E$. Moreover $Y'$ is $\frac{\eta\varepsilon}{1+\varepsilon}$-lc. Hence for $0\ll a< \tau$, $K_{(Y',\cG')_a}$ is big and $|m(K_\cG+a K_Y)|$ gives a map with two-dimensional image for all
\[
m\geq 8!\left (2+8\cdot \frac{1}{\varepsilon}\cdot \left\lfloor2\left (\frac{2}{\frac{\eta\varepsilon}{\varepsilon+1}}\right )^{ \frac{(2)^7}{ \frac{\eta\varepsilon}{\varepsilon+1}^5}  }\right\rfloor! \right )
\]
by Lemma \ref{lm: fano case}. Since $\varepsilon<\tau$, we can conclude that $\left|m\left(\frac{1}{\varepsilon}K_\cG+ K_Y\right)\right|$ gives a map with two-dimensional image for all
\[
m\geq 8!\left (2+8\cdot \frac{1}{\varepsilon}\cdot \left\lfloor2\left (\frac{2}{\frac{\eta\varepsilon}{\varepsilon+1}}\right )^{ \frac{(2)^7}{ \frac{\eta\varepsilon}{\varepsilon+1}^5}  }\right\rfloor! \right ).
\]
\end{proof}

\section{Applications}

\subsection{Bound on the degree of a general leaf in a non-isotrivial fibration on $\mP^2$}

Poincaré investigated the conditions under which a holomorphic foliation $\cF$ on $\mP^2$ admits a rational first integral—that is, when all the leaves of the foliation on $\mP^2$ are algebraic. 
In \cite{poincare1891lintegration}, Poincaré showed that it suffices to bound the degree of a generic leaf of $\cF$.  

In this section, we consider a foliation $\cF$ on a smooth surface $X$ that admits a rational first integral. 
Our aim is to refine the bound obtained in \cite[Theorem~A]{pereira}. 
Specifically, we provide an explicit bound for the degree of a general leaf of $\cF$ in terms of the degree of the foliation and the genus~$g$ of a general leaf.  

We focus on foliations of general type.

\begin{theorem}\label{th: bound degree}
Let $\cF$ be a foliation on $\mP^2$. 
Assume that $\cF$ is birationally equivalent to a non-isotrivial fibration of genus $g \geq 2$. 
Then the degree of the general leaf $F$ of genus~$g$ of $\cF$ satisfies
\[
  \deg F \leq 
  M_0 \left( \frac{1}{\tau_0} + 1 \right) (4g - 4) 
  \cdot \frac{1}{\tau_0} \deg \cF,
\]
where $M_0$ is the constant defined in Theorem~\ref{lm: bound M}, and $\tau_0$ is the constant determined in Theorem \ref{co: bound pseff}.
\end{theorem}

\begin{remark}
The genus of the general fibre of the foliation must appear as a parameter in the bound. 
In \cite[Example~4]{example}, the authors present an example in which they construct a family of foliations of general type with singularities of fixed local
analytic type and with first integrals of arbitrarily large degree. 
Hence, it is hopeless to bound the degree of an invariant curve solely by considering the local properties of the foliation.
\end{remark}

Theorem~\ref{th: bound degree} follows directly from the following proposition.

\begin{proposition}\label{prop: bound deg}
Let $\tau_0$ be the constant determined in Theorem~\ref{co: bound pseff}, 
and let $M_0(1,\tau_0)$ be the integer defined in Theorem~\ref{lm: bound M}.  
Let $\cF$ be a canonical foliation on a smooth surface $X$. 
Suppose that $\cF$ is a fibration, and let $F$ denote a general fibre of genus~$g$. 
If $K_\cF$ is big, then for every nef divisor $H$ we have
\[
  F \cdot H \leq 
  M_0\, C(g) \left( K_X + \frac{1}{\tau_0} K_\cF \right) \cdot H.
\]
Moreover, $C(g)$ satisfies the inequality
\[
  C(g) \leq \left(\frac{1}{\tau_0}+1\right)\cdot (4g-4).
\]
\end{proposition}

\begin{remark}
The proof of this proposition follows the same steps as in \cite[Proof of Theorem~5.4]{pereira}. 
However, instead of applying \cite[Proposition~5.1]{pereira}, Lemma~\ref{lm: bound M} provides the linear system 
\[
  |M(\tau K_X + K_\cF)|,
\]
which defines a rational map with two-dimensional image whenever the foliation is algebraically integrable.
\end{remark}

\begin{remark}
Theorem~\ref{th: bound degree} then follows, as in \cite[Proof of Theorem~A]{pereira}, by applying Proposition~\ref{prop: bound deg}.
\end{remark}

	\addcontentsline{toc}{section}{Bibliography}
    \bibliographystyle{alpha}
 	\bibliography{bib.bib}

\end{document}